\documentclass[12pt]{article}
\usepackage{amsmath,amssymb,amsthm}
\usepackage{fouriernc,graphicx}
\usepackage[top=1in,bottom=1in,left=1in,right=1in]{geometry}

\newtheorem{theorem}{Theorem}[section]
\newtheorem{lemma}[theorem]{Lemma}
\newtheorem{prop}[theorem]{Proposition}
\newtheorem{cor}[theorem]{Corollary}
\newtheorem{task}[theorem]{Task}
\newtheorem{deF}[theorem]{Definition}
\newtheorem{ques}[theorem]{Question}

\newtheorem{example}[theorem]{Example}

\newtheorem{notation}[theorem]{Notation}

\newcommand\VS[1]{\mbox{\rule{0pt}{#1}}}

\renewcommand{\implies}{\Rightarrow}

\newcommand{\zz}{\mathbb{Z}}
\newcommand{\nat}{\mathbb{N}}
\newcommand{\ratn}{\mathbb{Q}} 
\newcommand{\real}{\mathbb{R}}
\newcommand{\cplex}{\mathbb{C}}

\newcommand{\ncom}{\newcommand}

\def\abs#1{\left\vert #1 \right\vert}

\def\set#1{\lbrace #1 \rbrace}
\def\Set#1{\left\lbrace\, #1\, \right\rbrace}

\newcommand{\Inv}{{}^{-1}}

\newcommand{\al}{\alpha}
\newcommand{\ome}{\omega}
\newcommand{\ep}{\epsilon}
 \ncom{\Del}{\Delta}

\newcommand{\GGG}[1]{
\begin{gather*}
#1
\end{gather*}
}
\newcommand{\AAA}[1]{
\begin{align*}
#1
\end{align*}
}

\newcommand{\zec}{Zeckendorf}

\newcommand{\fib}{Fibonacci}
\newcommand{\funds}{fundamental sequence}
\newcommand{\lbd}{leading block distribution}

\newcommand{\HSW}[1]{\rule{#1\linewidth}{0pt}}

\newcommand{\cE}{\mathcal{F}}

\newcommand{\benlaw}{Benford's Law}
\newcommand{\cK}{\mathcal{K}}
\newcommand{\cH}{\mathcal{H}}
\newcommand{\cHst}{\mathcal{H}^*}

\newcommand{\fF}{\mathfrak{F}}
\newcommand{\fH}{\mathfrak{H}}
\newcommand{\fK}{\mathfrak{K}}

\newcommand{\vecb}{{\mathbf b}}
\newcommand{\vecc}{{\mathbf c}}

\newcommand{\lex}{lexicographical}

\newcommand{\cf}{coefficient function}

\newcommand{\UI}{\mathbf{I}}
\newcommand{\OI}{\UI}

\newcommand{\moD}[1]{\ (\mathrm{mod}\ #1)}

 \newcommand{\cF}{\mathcal{F}}
 \newcommand{\cFst}{\mathcal{F^*}}

\newcommand{\seq}[1]{\set{#1_n}_{n=1}^\infty}
\newcommand{\seQ}[1]{\set{#1 }_{n=1}^\infty}
\newcommand{\proB}[1]{\mathrm{Prob}\left\lbrace #1 \right\rbrace}
\newcommand{\wt}{\widetilde}

\newcommand{\msr}{\mathrm{msr}}

\newcommand{\len}{\mathrm{len}}

\newcommand{\FrcPart}[1]{\mathrm{frc}\left( #1 \right)}
\newcommand{\frc}{\FrcPart}
\newcommand{\prob}[1]{\mathrm{Prob}\Set{\VS{2ex}\,#1\,}}
\newcommand{\LB}{\mathrm{LB}}
\newcommand{\LD}{\mathrm{LD}}
\newcommand{\flr}[1]{\left\lfloor #1 \right\rfloor}
\newcommand{\nN}{n \in \nat}
\newcommand{\kN}{k \in \nat}

\newcommand{\wh}{\widehat}
\newcommand{\baR}{\overline}

\begin{document}

\begin{center}
\large\bf
Benford's Law under \zec\ expansion\\
\normalsize\rm
 Sungkon Chang and Steven J. Miller
\end{center}

 \begin{quote}
{\bf abstract}\quad
\footnotesize
 In the literature, 
Benford's Law is considered for base-$b$ expansions where $b>1$ is an integer.  
In this paper, we investigate the distribution of leading
\lq\lq digits\rq\rq\ of a sequence of positive integers under
other expansions such as \zec\ expansion, and declare
 what \benlaw\  should be under generalized \zec\ expansion.
\end{quote}

\section{Introduction}\label{sec:introduction}

Introduced in  \cite{benford,newcomb} is a  probability distribution of
the leading decimal digits of a sequence of positive integers,  known as {\it\benlaw}, and 
 the exponential sequences such as $\set{3^n}$ are standard examples 
 of sequences that satisfy \benlaw.
 Given $d\in\set{1,2,3,\dots,9}$,
the probability of having the leading digit $d$ in the decimal
expansion of $3^n$
is $\log_{10}\frac{d+1}{d}$,
and this distribution is  Benford's Law.
In fact, given a block $B$ of digits of any length,
the probability of having the leading block $B$ in the decimal
expansion of $3^n$ is given by a similar
logarithmic formula  as well, and 
this is known as {\it strong Benford's Law;} see Example \ref{exm:BL-10}. 
It is indeed a special property that a sequence has convergent proportions
for each leading digit.
For example,  
the proportion of odd integers $2n-1\le M$ with leading digit $d$ oscillates, and 
does not converge as $M\to\infty$; see Section \ref{thm:limsupinf}.

In the literature, 
Benford's Law is considered for base-$b$ expansions where $b>1$ is an integer. 
For example, the probabilities of 
the binary expansions of integer powers of $3$ having the leading binary digits
$100_2$  and $101_2$ are $\log_2\frac{2^2 + 1}{2^2}$ and 
$\log_2\frac{2^2 + 2}{2^2 + 1}$, respectively;
for  later reference, we may rewrite the values as follows:
\begin{equation}
\log_2\frac{1+2^{-2}}{1}\approx 0.322,\quad
\log_2\frac{1+2^{-1}}{1+2^{-2}}\approx 0.264.
\label{eq:binary-digits}
\end{equation}
In this paper, we shall consider the distribution of leading
\lq\lq digits\rq\rq\ of a sequence of positive integers under
other expansions such as \zec\ expansion  \cite{zec}.
For example, let $\seq F$ for $n\ge 1$ be the shifted
\fib\ sequence, i.e.,  
$F_{n+2}=F_{n+1}+F_n$ for all $\nN$ and $F_1=1$ and $F_2=2$,
and consider two \zec\ expansions:
 $3^5 = F_{12} + F_5 + F_2$
 and $3^8=F_{18}+F_{16}+F_{14}+F_{11}+F_{7}+F_{5}$.
Similar to the way the binary expansions are denoted, we may write
$$ 3^5=100000010010_F,\quad
3^8=101010010001010000_F$$
 where $1$'s are inserted at the $k$th place from the right if 
 $F_k$ is used in the expansions.
\begin{deF}\rm \label{def:LB-F}
Let $A=\set{0,1}$.
Given $\set{s,n}\subset\nat$, 
let $n=\sum_{k=1}^M a_k F_{M-k+1}$ be 
  the \zec\ expansion of $n$ (where $a_1=1$).
We define $\LB_s(n):= (a_1,\dots,a_s)\in A^s$
  if  $M\ge s$; otherwise, $\LB_s(n)$ is undefined.
The tuple $\LB_s(n)$ is called {\it the leading block of   $n$ with length $s$ under 
\zec\ expansion}.
\end{deF}
\noindent
For example, $\LB_3(3^5)=(1,0,0)$,
$\LB_3(3^8)=(1,0,1)$, and $\LB_6(3^8)=(1,0,1,0,1,0)$.
Since $\LB_2(n)=(1,0)$ for all integers $n\ge 2$, 
it is only meaningful to consider the first three or more \zec\ digits. 
We prove Theorem \ref{thm:BL-length-1} in this note.
\begin{deF}
\rm
Given  a conditional statement 
$P(n)$ where $\nN$, and a subset $A$ of $\nat$, let us define 
\GGG{
\prob{n\in A: P(n)\text{ is true}}
:=\lim_{n\to \infty} \frac{
\#\set{k\in A : P(k)\text{ is true},\ k\le n}}
{\#\set{k\in A :   k\le n}}. 
}
\end{deF}
\noindent
For example,  if $A=\set{n\in \nat : n\equiv 2 \mod 3}$,
then 
$\prob{ n \in A : n \equiv 1 \mod 5}=\frac15$.
If $A$ is finite, the limit always exists.

Let $\phi$ be the Golden ratio.
The following is an analogue of Benford's Law under  binary expansion 
demonstrated in (\ref{eq:binary-digits}).
\begin{theorem}\label{thm:BL-length-1}
Let $a>1$ be an integer.
\AAA{
\prob{\nN : \LB_3(a^n)=(1,0,0) }
&\ =\ \log_\phi(1+\phi^{-2})\approx .672,\\
\prob{\nN : \LB_3(a^n)=(1,0,1) }
&\ =\ \log_\phi\frac{\phi}{1+\phi^{-2}}\approx .328.
} 
\end{theorem} 
\noindent
In particular, they exist!
Although the probabilities are different from the binary cases, 
the structure of the log expressions in Theorem 
\ref{thm:BL-length-1} is quite similar to that of 
the binary expansions in (\ref{eq:binary-digits}), i.e.,
the denominators of the quotients express the leading digits in power expansions
with respect to their bases.
The exponential sequences $\seQ{a^n}$ where $a>1$ is an integer
  are standard sequences that satisfy 
  \benlaw\ under base-$b$ expansion.
Motivated from these standard examples, 
we define \benlaw\ under \zec\ expansion to be the above
distribution of the leading blocks $(1,0,0)$ and $(1,0,1)$ under \zec\ expansion; see Definition \ref{def:F-benford-law}.

The exponential sequences $\seQ{a^n}$  
  are standard sequences for so-called {\it strong
  \benlaw\ under base-$b$ expansion} as well;
  see Example \ref{exm:BL-10}.
  We  introduce below  the probability of  the leading \zec\ digits of $a^n$ with arbitrary length, which is
 a generalization of Theorem \ref{thm:BL-length-1};
 this result is rewritten in Theorem \ref{thm:equidistribution-examples} with more compact notation.
\begin{deF}
\rm \label{def:B-tuples}
Let $A=\set{0,1}$, and let $s\ge 2$ be an integer.
Let $\mathbf b=(b_1,b_2,\dots, b_s)\in A^s$ such that $b_1=1$ and 
$b_k b_{k+1}=0$ for all $1\le k\le s-1$.
We define $\wt\vecb$ to be a tuple $ (\wt b_1,\dots,\wt b_s)\in A^s$ as follows.
If $1+\sum_{k=1}^s b_k F_{s-k+1}<F_{s+1}$,
then $\wt b_k$ for $1\le k\le s$ are defined to be  integers in $A$ such that 
$1+\sum_{k=1}^s b_k F_{s-k+1}= \sum_{k=1}^s\wt  b_k F_{s-k+1}$
and $\wt b_k\wt  b_{k+1}=0$ for all $1\le k\le s-1$.
If $1+\sum_{k=1}^s b_k F_{s-k+1} =  F_{s+1}$,
then $\wt b_1:=\wt b_2:=1$, and $\wt b_k:=0$ for all $3\le k\le s$.
\end{deF}
\noindent
For the case of $1+\sum_{k=1}^s b_k F_{s-k+1}<F_{s+1}$,
the existence of the tuple $\wt \vecb$ is guaranteed by 
\zec's Theorem.

 \begin{theorem}
Let $a>1$ and $s\ge 2$ be integers. 
Let $\vecb$ and $\wt \vecb$ be tuples defined in Definition \ref{def:B-tuples}.
 Then,
$$\prob{\nN : 
\LB_s(a^n)=\vecb}
\ =\ \log_\phi\,\frac
{\sum_{k=1}^s \wt b_k \phi^{-(k-1)}}
{\sum_{k=1}^s b_k \phi^{-(k-1)}}. $$
 
\end{theorem} 
\noindent
For example,
\begin{align*}
\prob{\nN : 
\LB_6(a^n)=(1,0,0,0,1,0)}
&\ =\ \log_\phi\frac{1+\phi^{-3}}{1+\phi^{-4}}\approx 0.157\\
\prob{\nN : 
\LB_6(a^n)=(1,0,1,0,1,0)}
&\ =\ \log_\phi\frac{1+\phi^{-1}}{1+\phi^{-2}+\phi^{-4}}\\
&\ =\ \log_\phi\frac{ \phi }{1+\phi^{-2}+\phi^{-4}}\approx 0.119.
\end{align*}

\noindent
 As in \benlaw\ under \zec\ expansion, we define the  probability distributions described in 
 Theorem    \ref{thm:equidistribution-examples} 
 to be 
 {\it strong \benlaw\ under \zec\ expansion}; see Definition \ref{def:strong-benford}.
 
Exponential sequences are standard examples for \benlaw s, but some exponential sequences 
do not satisfy 
\benlaw\ under some base-$b$ expansion.
Let us demonstrate examples under \zec\ expansion.
Let $\seq G$ be the sequence given by $G_k = F_{2k} + F_k$ for $k\in \nat$.
Then, given an integer $s>1$, the $s$ leading \zec\ digits of $G_k$ is 
$100\cdots00_F$ as $k\to\infty$ since the gap $2k-k=k$ between the indices of $F_{2k}$ and $F_n$ approaches $\infty$.
Thus,
$\prob{\nN : \LB_s(G_n)=(1,0,0,\dots,0) }=1$ for all $s\in\nat$, and the probabilities of other digits of length $s$ are all (asymptotically) $0$. 
Similar probability distributions occur for the Lucas sequence $\seq K$ given by $K_{k+2}=K_{k+1} + K_k$ for $\kN$
and $(K_1,K_2)=(2,1)$.
Given $s\in\nat$,
  the probabilities of having    leading \zec\ digits of length $s$ are entirely concentrated on 
  one particular string of digits.
  For example, 
  $\prob{\nN : \LB_{10}(K_n)=
  (1,0,0,0,1,0,0,0,1,0) }=1$, and the probabilities of having other digits of length $10$ 
is all (asymptotically)  $0$; see Example \ref{exm:lucas} for full answers.

Generalized \zec\ expansions are introduced in \cite{chang,mw}.
In Section \ref{sec:general-\zec},
we prove Theorem \ref{thm:equidistribution-examples-2} on the probability  of the leading digits of $a^n$ with arbitrary length under 
generalized \zec\ expansion, and define these probability distributions to be
 strong \benlaw\ under generalized \zec\ expansion; see Definition \ref{def:BL-GZ}.
As in the concept of 
{\it absolute normal numbers} \cite{erdos}, we introduce 
in Definition \ref{def:BL-absolute}  
the notion of {\it absolute \benlaw}, which is the property of
satisfying strong \benlaw\ under all generalized \zec\ expansions.
For example, the sequence given by $K_n=\flr{ \frac{\phi}{\sqrt5}(\tfrac{89}{55})^{n } } $
for $\nN$ satisfies strong \benlaw\ under
all generalized \zec\ expansions; see Example \ref{exm:ABL}.
Its first fifteen values are listed below:
$$(1, 1, 3, 4, 8, 12, 21, 34, 55, 89, 144, 233, 377, 610, 988).$$
They  are nearly equal to the \fib\ terms as 
$\frac{89}{55}$ is the $10$th convergent of the continued fraction of $\phi$.
The differences amplify as we look at higher terms, and even under \zec\ expansion, 
this sequence satisfies strong \benlaw.

It is also natural to consider sequences that have different
distributions, and in this note we  investigate 
  other distributions of leading digits  under generalized \zec\ expansions 
  as well.  In the following paragraphs, we shall explain this approach using base-$10$ expansion.
  The results for other expansions are introduced in Section \ref{sec:Other-continuations} and \ref{sec:general-\zec}.

Strong Benford's Law for the sequence $\seQ{3^n}$ under decimal expansion  follows from the equidistribution of the fractional part
of $\log_{10}(3^n)$ on the interval $(0,1)$.
 We realized that the function $\log_{10}(x)$ is merely a tool
 for calculating the leading digits, and 
 that other distributions of leading digits naturally emerge
  as we modified the function $\log_{10}(x)$.

We noticed that the frequency of leading digits
  converges when a continuation of the sequence $\seQ{10^{n-1}} $ has convergent behavior over the intervals
  $[n,n+1]$, and we phrase it more precisely below.
\begin{deF} \label{def:uniform-continuation} \rm
Let $\seq H$ be  an increasing sequence of positive integers.
A continuous function $h : [1,\infty)\to \real$ is called 
a {\it uniform continuation of $\seq H$} if $h(n)=H_n$ for all $\nN$, and 
the following sequence of functions $h_n: [0,1] \to [0,1]$
uniformly converges to an increasing (continuous) function:
$$h_n(p)=\frac{h(n+p)-h(n)}{h(n+1)-h(n)}.$$
If $h$ is a  uniform continuation of  $\seq H$,
 let $h_\infty : [0,1] \to [0,1]$ denote the increasing continuous  function given by
 $h_\infty(p)=\lim_{n\to\infty}h_n(p)$.
\end{deF}

Theorem \ref{thm:continuation-10} below is a version specialized for decimal expansion.
The proof of this theorem is similar to, and much simpler than the proof of Theorem \ref{thm:continuation} for \zec\ expansion, and 
we leave it to the reader.  

\begin{deF}
\rm \label{def:fractional-part}
If $\al\in\real$, we denote the fractional part of $\al$ 
by $\frc \al$.
Given a sequence $\seq K$ of real numbers,
we say, {\it $\frc{K_n}$ is equidistributed}
if $\prob{\nN : \frc{K_n}\le \beta} = \beta$ for all $\beta\in[0,1]$.
\end{deF}
 For example,
 consider the sequence $\seQ{\frc{n\pi}} $ where $\pi\approx 3.14$ is the irrational number.
 Then, by Weyl's Equidistribution Theorem, $\frc{n\pi}$ is equidistributed on the interval $[0,1]$.
The sequence $\seQ{\sin^2(n)}$ is an example of 
sequences that have $\prob{\nN : \sin^2(n)\le \beta}$
defined for each $\beta\in[0,1]$, and 
the probability is $\frac1\pi\cos\Inv(1-2\beta)$.
Thus, it is not equidistributed on $[0,1]$.

\begin{theorem}\label{thm:continuation-10}
Let $h : [1,\infty) \to \real$ be  a  uniform continuation of the sequence $\seQ{10^{k-1}} $.
    Then,  there is a  sequence
    $\seq K$ of positive integers  approaching $\infty$  (see Theorem \ref{thm:continuation-H} for the 
description of $K_n$) 
 such that 
    $\frc{ h\Inv(K_n) }$ is equidistributed. 

Let $\seq K$ be a sequence of positive integers  approaching $\infty$ such that 
    $\frc{ h\Inv(K_n) }$ is equidistributed.
Let $d$ be a positive integer of $s$ decimal digits.
Then, the probability of the $s$ leading decimal digits of $K_n$ being $d$ is 
equal to
\GGG{ 
h_\infty\Inv
\left(\frac{(d+1)- 10^{s-1}}{9\cdot 10^{s-1}} \right) - 
 h_\infty\Inv\left( \frac{d- 10^{s-1}}{9\cdot 10^{s-1}}  \right) .
}

\end{theorem} 
\begin{example}\label{exm:BL-10}
\rm
Let $h : [1,\infty) \to \real$ be the function given by
$h(x)=10^{x-1}$.
Then, $h$ is a  uniform continuation of the sequence $\set{10^{n-1}}$, 
and $h_\infty(p)=\frac19(10^p-1)$.
 By Theorem \ref{thm:continuation-H},
 the sequence $\seq K$ with the equidistribution property 
 is given by $K_n = \flr{ 10^{n + \frc{n\pi}} }$, but 
 there are simpler sequences such as $\seQ {3^n}$ that have 
 the property.

By Theorem \ref{thm:continuation-10}, the probability of the $s$ leading decimal digits of $K_n$ being $d$ is equal to 
$$\log_{10}\frac{d+1}{10^{s-1}}-\log_{10}\frac{d }{10^{s-1}}
\ =\ \log_{10}\left( 1+\dfrac1d\right) $$
where $d\in\nat$ has $s$ decimal digits.
This distribution is known as strong \benlaw\ under 
base-$10$ expansion, and we may say that 
strong \benlaw\ under base-$10$ expansion arises from the logarithmic continuation of $\seQ{10^{n-1}}$.
For this reason, we call $h(x)$ {\it a Benford continuation of the base-10 sequence}.
\end{example}

\begin{example}\label{exm:log-10}\rm
 Let $h : [1,\infty) \to \real$ be the function whose 
graph is the union of the line segments 
from $(n,10^{n-1})$ to $(n+1,10^n)$ for all $\nN$.
Let $\seq K$ be the sequence given by $K_n= \flr{10^{n + \log_{10}( 9{\frc{n\pi}+1})}}$ as described in 
Theorem \ref{thm:continuation-H}.
Then, the fractional part $\frc{h\Inv(K_n)}$ is equidistributed.
The limit function $h_\infty$ defined in Theorem \ref{thm:continuation-10} is given by $h_\infty(p)=p$ for $p\in[0,1]$, and 
 given a decimal expansion $d$ of length $s$, the probability of 
 the $s$ leading decimal digits of $K_n$ being $d$ is (uniformly) equal to $1/(9\cdot10^{s-1})$ by 
 Theorem \ref{thm:continuation-10}.

The first ten values of  $ {K_n}$   are 
\footnotesize
\GGG{
(22, 354, 4823, 60973, 737166, 8646003, 99203371, 219467105, \
3469004940, 47433388230).
}\normalsize
 For example, if we look at many more terms of $ K$,
 then the first two digits $ 22 $ of $ K_1$ will occur as leading digits
 with probability $1/90\approx 0.011$, and the probability for the digits $99$
 is also $1/90$.
As in constructing a normal number, it's tricky to 
 construct a sequence of positive integers with this property,
 and prove that it has the property.
Let us note here that the $s$ leading decimal digits of the sequence $\seQ{n}$  has frequency close to $1/(9\cdot10^{s-1})$, but it oscillates and does not converge as more terms are considered; 
 see Theorem \ref{thm:limsupinf} for a version under \zec\ expansion.
 In Example \ref{exm:zec-uniform}, we demonstrate the \lq\lq line-segment\rq\rq\ continuation of the \fib\ sequence.
\end{example}
In Example \ref{exm:fake-BF}, we use a more refined 
\lq\lq line segment continuation\rq\rq,
and demonstrate a uniform continuation that generates 
the distribution of leading blocks that satisfies strong \benlaw\ up to the $4$th digits, but 
does not satisfy the law for the leading blocks of length $>4$.
 
Theorem \ref{thm:continuation-10}  suggests
that given a uniform continuation  $h$ of the sequence $\seQ{10^{n-1}}$,
  we associate certain distributions of leading digits, coming from the equidistribution property.
It's natural to consider the converse that given
a sequence $\seQ{K_n}$ with \lq\lq continuous distribution of leading digits\rq\rq\ of arbitrary length,
 we associate a certain uniform continuation of $\seQ{10^{n-1}}$.
 Theorem \ref{thm:converse-10} below is a version for base-$10$
 expansion.
 In Section \ref{sec:Other-continuations},  we introduce results on this topic for 
 the \fib\ sequence $\seQ{F_n}$.
 The proof of Theorem \ref{thm:converse-10}  is similar to, and simpler than Theorem 
 \ref{thm:f-star} for the \fib\ expansion, and
 leave it to the reader.
 
 \begin{theorem}\label{thm:converse-10}
Let $\seq K$ be a sequence of positive integers approaching $\infty$.
Let $h_K^* : [0,1]\to[0,1]$ be the function given by 
$h_K^*(0)=0$, $h_K^*(1)=1$, and 
\begin{equation}
h_K^*(\tfrac19(\beta - 1))\ =\ \lim_{s\to\infty}
\prob{ \nN :
\text{The $s$ leading decimal digits of $K_n$ is 
$\le \flr{10^{s-1}\beta}$ } }  \label{eq:f-star-10}
\end{equation} 
where $\beta$ varies over the real numbers in the interval
$[1,10)$ and we assume that the RHS of 
(\ref{eq:f-star-10}) is defined for all $\beta\in[1,10)$.
If $h_K^*$ is an increasing continuous function,
then there is a uniform continuation $h$ of the sequence
$\seQ{10^{n-1}}$ such that 
$h_\infty \Inv= h_K^*$,
and $\frc{h\Inv(K_n)}$ is equidistributed.

\end{theorem} 

If a sequence
    $K$ of positive integers  approaching $ \infty$ satisfies (\ref{eq:f-star-10})
    where $h_K^*$ is an increasing continuous function,
    the sequence is said to {\it
    have continuous \lbd\ under base-$10$ expansion}; see 
    Definition \ref{def:f-star} for \zec\ expansion.
By Theorem \ref{thm:continuation-10} and \ref{thm:converse-10}, 
we have
\begin{theorem}
A  sequence
    $\seq K$ of positive integers   approaching $ \infty$ has continuous \lbd\ under base-$10$ expansion
    if and only if $\frc{h\Inv(K_n)}$ is equidistributed
    for some uniform continuation $h$ of $\seQ{10^{n-1}}$.
\end{theorem} 
\begin{cor}
If a sequence $\seq K$ satisfies strong \benlaw\ under base-$10$ expansion,
then $\frc{\log_{10}(K_n)}$ is equidistributed.
\end{cor}

It is interesting that 
 \benlaw\ under base-$b$ expansion arises  within the \zec\ expansion of 
a positive integer, i.e., if we randomly select
an integer $n$ in $[1, F_{m})$,
then the frequency of 
   the \fib\ terms with leading decimal digit $d$ 
   among the summands of the \zec\ expansion of $n$
   is nearly $\log_{10}(1+\frac1d)$ for almost all $n\in [1, F_{m})$.
 This is proved in  \cite{miller-behavior}.
In fact, they prove a result  for attributes that are far more general than leading digits, and
 the result holds for all generalized \zec\ expansions as well.

 Their result immediately applies to a general setup where a generalized \zec\ expansion and base-$10$ expansion are 
 replaced with two arbitrary generalized \zec\ expansions.
 The full statement is found in Theorem 	\ref{thm:Psmu}, and below
 we introduce a specialized version for the binary and \zec\ expansions.
 \begin{theorem} \label{thm:benford-behavior}
Let $S=\set{ \nN : \LB_3(2^{n-1})=(1,0,1)}$, and let $t\in\nat$.
Given $n\in[1,2^t)$,
let $n=\sum_{k\in A_n} 2^k$ be the binary expansion of $n$ where $A_n$ is a finite subset of $\nat$, 
and define $P_t(n)=\#(A_n\cap S)/\#A_n$.

Then, given a   real number 
$\ep>0$, 
the probability of $n\in [1,2^t)$
such that 
$$\abs{P_t(n)-\log_\phi\frac{ \phi}{1+\phi^{-2}}}\ <\  \ep $$ is equal to $1+o(1)$
as a function of $t$.
 \end{theorem}
\noindent
Notice that by Theorem \ref{thm:BL-length-1},  $\prob{\nN : n\in S}= \log_\phi\frac{ \phi}{1+\phi^{-2}}$, and hence,
we may say that \benlaw\ under \zec\ expansion arises  within the binary expansion of positive integers.

The remainder of this paper is organized as follows.
In Section \ref{sec:notation}, the notations for sequences and \cf s are introduced.
In Section \ref{sec:benford}, the distribution of leading blocks of exponential sequences under \zec\ expansion is introduced, and 
 \benlaw\ and strong \benlaw\ under \zec\ expansion are declared.
 Introduced in Section \ref{sec:calculations} 
 are the method of calculating the distribution results introduced in Section \ref{sec:benford}, and 
  also the distribution results for monomial sequences $\seQ{n^a}$.
In Section \ref{sec:Other-continuations},
 we introduce a general approach to  the distributions of leading blocks under \zec\ expansion that are 
 different from that of \benlaw.
 The approach establishes the correspondence between the continuations of the \fib\ sequences and the distributions of leading blocks
 under \zec\ expansion.
 In Section \ref{sec:general-\zec},
  we introduce definitions and results that generalize the contents of Sections \ref{sec:benford}, \ref{sec:calculations},
  and \ref{sec:Other-continuations} for generalized \zec\ expansions.
  The absolute \benlaw\ mentioned earlier in this section is properly introduced in Section \ref{sec:general-\zec} as well.
In Section 	\ref{sec:within-expansion},
the Benford behavior introduced in Theorem \ref{thm:benford-behavior} is generalized for the setting of two generalized 
\zec\ expansions.

\section{Notation and definitions}\label{sec:notation}

\begin{notation}\rm
\label{def:N0}
Let $\nat_0:=\nat\cup\set{0}$, and let $\Omega_n :=
\set{k\in \nat : k\le n}$.
For simpler notation, let us use a capital letter for a sequence of numbers, and 
use the infinite tuple notation for listing its values, e.g.,
$Q=(2,4,6,8,\dots)$.
We use the usual subscript notation for individual values, e.g., $Q_3=6$.
\end{notation}

\begin{deF}\label{def:products}\rm 
Tuples $(c_1,c_2,\dots,c_t)\in\nat_0^t$ where $t\in \nat$ are called
{\it \cf s of length $t$}
if $c_1>0$.
If  $\ep$ is a \cf\ of length $t$, we denote the $k$th entry by $\ep(k)$ (if $k\le t$), and its length $t$ by $\len(\ep)$.
For a \cf\  $\ep$, let $\ep *  Q$ denote $\sum_{k=1}^t \ep(k) Q_{t-k+1}$ where $t=\len(\ep)$, and 
let $\ep \cdot Q$ denote  $\sum_{k=1}^t \ep(k) Q_{k}$.

\end{deF}
\noindent
If  $\ep=(4,1,6,2)$ and $Q$ is a sequence, then 
$\ep * Q = 4 Q_4 + Q_3 + 6 Q_2 + 2 Q_1$, and 
$\ep \cdot Q = 4 Q_1 + Q_2 + 6 Q_3 + 2 Q_4$.

\section{Benford's Law for \zec\ expansions}\label{sec:benford}

 Let $a$ and $b$ be two integers $>1$ such that $\gcd(a,b)=1$.
The  sequence  $K$ be the sequence given by $K_n=a^n $  
is a  standard example of sequences that satisfy 
Benford's Law under base-$b$ expansion.
  We shall declare the behavior of the leading digits 
  of the \zec\ expansion of $a^n$ to be Benford's 
  Law under \zec\ expansion.

Let us begin with formulating \zec's Theorem 
  in terms of \cf s.
\begin{deF} \label{def:ome}
\rm
Let $\cE$ be the set of \cf s $\ep$ such that $\ep(k)\le 1$ for all $k\le \len(\ep)$, and $\ep(k) \ep (k+1)=0$   all $k\le \len(\ep)-1$.
Let $F$ be the shifted Fibonacci sequence such that 
$F_{n+2}= F_{n+1} + F_n$ for all $\nN$ and $(F_1,F_2)
=(1,2)$.
Let $\phi$ be
 the golden ratio, let $\ome:=\phi^{-1}$,
 and let $\wh F=(1,\ome,\ome^2,\dots)$ be the sequence given by $\wh F_n=\ome^{n-1  }$. 
\end{deF} 

Recall the product notation from Definition \ref{def:products}.
\begin{theorem}[\cite{zec}, \zec's Theorem]
For each positive integer $n$,  
  there is a unique \cf\ $\ep\in \cE$ such that 
$n=\ep*F$.
\end{theorem} 
\noindent
Recall the example  $3^5 = F_{12} + F_5 + F_2$.
If  $\ep=(1,0,0,0,0,0,0,1,0,0,1,0)$,
then $\ep\in \cF$ and $3^5 = \ep * F$.

\begin{deF}
\rm \label{def:F-expansion}
The expression $n=\ep * F$ where $n\in\nat$ and $\ep\in\cF$ is called 
{\it the $\cE$-expansion of $n$} or
{\it the \zec\ expansion of $n$}.
\end{deF}

\subsection{\benlaw}
If $\ep\in\cE$ and $\len(\ep)\ge 2$, then $(\ep(1),\ep(2))=(1,0)$ is always the case, and hence,
the probability of having  $(\ep(1),\ep(2))=(1,0)$ is $1$.
For the purpose of demonstration, we consider the first three entries of $\ep$.

To denote arbitrarily many {\it leading blocks of \cf s}, which are defined in Definition \ref{def:F-blocks} below,
 we shall use the boldface font  and subscripts, e.g.,   $\vecb_1$ and $\vecb_2$, and 
 in particular, $\vecb_k$ for $k=1,2$ are not numbers, but tuples.
 The reader must not be confused with the entries of a sequence $Q$, e.g., $Q_1$ and $Q_2$, which are numbers, and we use the regular font for sequences.

\begin{deF} \rm \label{def:F-blocks}
A \cf\ of length $s$ is also called 
{\it a leading block of length $s$} in the context of 
  investigating the frequency of leading blocks,
  and it is denoted with boldface fonts, e.g. $\vecb=(1,0,0,1)\in\cF$,
  $\vecb(3)=0$, and 
  $\vecb(4)=1$.
  Let $\cF_3:=\set{\vecb_1,\vecb_2}$ where 
  $ \vecb_1=(1,0,0),\ \vecb_2=(1,0,1)$ are leading blocks of length $3$, and 
  the set is called 
  {\it the set of leading blocks of length $3$ under $\cF$-expansion}.
  If $\vecb \in \cF_3$ and $\vecb =\vecb_1$, then define $ \wt \vecb:=\vecb_2$, and 
  and if $\vecb \in \cF_3$ and $\vecb =\vecb_2$, then 
  define $ \wt \vecb:= (1,1,0)$.
 
\end{deF}

 The block $\wt \vecb = (1,1,0)$ is not a member of $\cF$, and hence, 
 does not occur as the leading block of an $\cF$-expansion,
 but it's convenient to use for Theorem \ref{thm:fib-distribution},
 where  we rely on the equality $\wt \vecb\cdot (1,\ome^{1},\ome^{2}) = \phi$;
 see Definitions \ref{def:products} and \ref{def:ome}.   
 The block $\wt \vecb$ makes  the statements of Definition \ref{def:F-benford-law} below more aesthetic, 
and the principle of defining an exclusive block such as $ (1,1,0)$ for other generalized \zec\ expansions  will be explained in
 Definition \ref{def:B-tilde} and Section \ref{sec:general-\zec}.

The following is a special version of Corollary 	\ref{cor:equidistribution-examples}, and 
it is Theorem \ref{thm:BL-length-1} written in terms of the dot product and blocks.
Recall the notation $\LB_s$ from Definition \ref{def:LB-F},  
the set $\cF_3$ from Definition \ref{def:F-blocks},  
the sequence $\wh F$ from Definition \ref{def:ome}, and the dot product
from Definition \ref{def:products}.
\begin{theorem}\label{thm:fib-distribution}
Let $K$ be a sequence given by $K_n=a^n$ where $a>1$ is an integer.
Then, given $\vecb\in \cF_3$, 
$$
\prob{n \in \nat : \LB_3(K_n)=\vecb }\ =\ \log_\phi
  	\frac{\wt \vecb\cdot \wh F}{ \vecb \cdot \wh F  }.
  	$$
\end{theorem} 
 
Motivated from the distribution of these standard sequences, we introduce the following definition. 

\begin{deF}\label{def:F-benford-law}\rm 

 A sequence $K$ of positive integers is said to {\it satisfy $\cE$-Benford's Law}
  or {\it satisfy Benford's Law under   $\cE$-expansion}
  if given $\vecb\in \cF_3$, 
$$
\prob{n \in \nat : \LB_3(K_n)=\vecb}\ =\ \log_\phi
  	\frac{\wt \vecb\cdot \wh F}{ \vecb \cdot \wh F  }.
  	$$
\end{deF}
\noindent
Let us demonstrate how the structure of the formulas in Definition \ref{def:F-benford-law} compares 
with   the one for base-10 expansion.
Consider the two leading blocks $\vecc_1=(2,1,2)$ and $\vecc_2=(2,1,3)$ for base-$10$ expansion. Let $b=10$.  
Then, strong Benford's Law for decimal expansion  requires the probability of having the leading block $\vecc_1$
to
be $\log_{10}\frac{213}{212}$, which is equal to  
\GGG{
\log_b\frac{ \vecc_2\cdot (1,b^{-1},b^{-2})}{ \vecc_1\cdot (1,b^{-1},b^{-2})}
\ =\ 
\log_b\frac{b^2 \vecc_2\cdot (1,b^{-1},b^{-2})}{b^2 \vecc_1\cdot (1,b^{-1},b^{-2})}
\ =\ \log_b\frac{ \vecc_2\cdot (b^2,b,1)}{ \vecc_1\cdot(b^2,b,1)}\ =\ \log_{10}\frac{213}{212}.
} 
The first expression in terms of the negative powers of $b$ is analogous to the ones 
in Definition \ref{def:F-benford-law}.
 
\subsection{Strong \benlaw}\label{sec:strong-benford}
Under base-$b$ expansion,
  a sequence $K$ is said to satisfy strong Benford's Law if 
  the probability of the first $M$ leading digits of $K_n$ 
  satisfies a certain logarithmic distribution, and 
   exponential sequences $\seQ{a^n}$ where $a>1$ is an integer
  are standard examples   that 
satisfy strong Benford's Law under base-$b$ expansion.
In Corollary \ref{cor:equidistribution-examples}, 
we calculate  the 
 distribution of leading blocks of arbitrary length of 
  the \zec\ expansions of exponential sequence $\seQ{a^n}$.
We declare this distribution to be 
 {\it strong \benlaw\ under \zec\ expansion}.
 We 
 state the formal definition below.

 Recall the convolution $*$ from Definition \ref{def:products}.
\begin{deF}\label{def:B-tilde}\rm 
 
Given an integer $s\ge 2$, let
 $\cF_s :=\set{\vecb_1,\vecb_2,\dots, \vecb_\ell}$ be the finite set of 
 the leading  blocks of length $s$ occurring in the $\cF$-expansions of the positive integers such that
  $1+\vecb_k*F=\vecb_{k+1}*F$ for all $k\le \ell-1$. The leading block $\vecb_\ell$ is called 
  {\it the largest leading block of length $s$ under $\cF$-expansion}.
  
  If $s$ is even, then let $\vecb_{\ell+1}:= (1,0,1,0,\dots,1,0,1,1)$, and 
if $s$ is odd, then it  is  $\vecb_{\ell+1}:= (1,0,1,0,\dots,1,1,0)$.
 If $\vecb=\vecb_k\in \cF_s$, then we denote $\vecb_{k+1}$ by $\wt \vecb$.

\end{deF}
 \noindent 
Notice that the existence of $\wt \vecb$ defined above is guaranteed by 
\zec's Theorem.
Let us demonstrate   examples of $\vecb$ and $\wt \vecb$.
Let $\vecb=(1,0,0,0,1,0)\in\cF_6$.  Then, $\wt \vecb=(1,0,0,1,0,0)\in\cF_6$, and 
$1+\vecb*F = \wt \vecb*F$.
If we list the \cf s  in $\cF_6$ with respect to the \lex\ order,
then $\wt \vecb$ is the immediate successor of $\vecb$ if $\vecb\ne (1,0,1,0,1,0)$.

For each case of $s$ being even or odd,  
the largest leading block $\vecb$ of length $s$ satisfies $1+ \vecb * F = \wt \vecb * F$.
If $\vecb'=(1,0,1,0,1,0)$,
then  $\wt \vecb'=(1,0,1,0,1,1)$, and below we shall demonstrate that 
the  equality $\wt \vecb' \cdot \wh F =\sum_{k=0}^2 \ome^{ 2k} + \ome^{ 5} = \phi$
makes the sum of the probabilities in Theorem 
\ref{thm:equidistribution-examples} and 
Definition \ref{def:strong-benford} be $1$.

Let us compare this setup with the case of base-$10$ expansion.
Let $\vecc=(4,5,6,7,8,9)$ be the leading block of length 6
for base-$10$ expansion, and 
let  the sequence $H$ given by $H_n= 10^{n-1} $ be the \lq\lq base\rq\rq\ sequence.
Then, $1 + \vecc * H = \wt \vecc * H$ where 
$\wt \vecc =  (4,5,6,7,9,0)$.
If we list all the \cf s of length 6, with respect to the \lex\ order, that are legal for base-10 expansion, 
then $\wt \vecc$ is  the immediate successor of $\vecc$.
If $\vecc'=(9,10,9,9,9,9)$, then we let $\wt \vecc'=(9,10,0,0,0,0)$,
and $\sum_{n=1}^6 \wt \vecc'(n) 10^{n-1}=1+\vecc'*H=10^6$.
If strong \benlaw\ under base-10 expansion is satisfied,
the probability of having the leading block $\vecc'$ 
under base-10 expansion is 
$$\log_{10} 
\frac{\wt \vecc'*   H}{\vecc' * H}\ =\ \log_{10} 
\frac{\wt \vecc'\cdot \wh H}{\vecc' \cdot \wh H}\ =\ 1 - \log_{10}\vecc'\cdot \wh H
$$
where $\wh H$ is the sequence given by $\wh H_n = 10^{-(n-1)}$.

Recall the sequence $\wh F$ from Definition \ref{def:ome}.
\begin{theorem}\label{thm:equidistribution-examples}
Let $K$ be a sequence of positive integers given by
 $K_n= a b^n(1+o(1))$ where $a$ and $b$ are positive real numbers such that $\log_\phi b$ is irrational.
Then,  given  $\vecb\in\cF_s$ where $s\ge 2$,
$$\prob{\nN : \LB_{ s}(K_n) = \vecb} 
\ =\  \log_\phi\frac{ \wt \vecb \cdot \wh F}{  \vecb \cdot \wh F}.$$
\end{theorem} 

\begin{proof}
It follows immediately from Corollary \ref{cor:equidistribution-examples}.
\end{proof}
Let us demonstrate below that the probabilities add up to $1$ for $s=6$, but 
the argument is sufficiently general to be extended for all cases of $s$.
Let $ \cF_6=\set{\vecb_1,\dots,\vecb_\ell}$  such that 
$\vecb_{k+1} = \wt \vecb_k$ for all $1\le k\le \ell$.
Then, $\vecb_1=(1,0,0,0,0,0)$ and 
$\vecb_\ell=(1,0,1,0,1,0)$.
Then, $\vecb_{\ell+1}=(1,1,0,0,0,0)$, and 
\GGG{
\sum_{k=1}^\ell  \log_\phi\frac{ \wt \vecb_k \cdot \wh F}{  \vecb_k \cdot \wh F}
\ =\ \sum_{k=1}^\ell  \log_\phi(  \vecb_{k+1} \cdot \wh F)
-  \log_\phi(\vecb_k \cdot \wh F)\ =\ \log_\phi(\vecb_{\ell+1} \cdot \wh F)
-\log_\phi 1\ =\ 1.
}

\begin{deF}\label{def:strong-benford}
\rm
 Let $K $ be a   sequence of positive integers approaching $\infty$.  
Then, $K$ is said to {\it satisfy strong Benford's Law under   $\cE$-expansion}
if given  $\vecb\in\cF_s$ where $s\ge 2$,
$$\prob{\nN : \LB_{s}(K_n) = \vecb} 
\ =\ \log_\phi\frac{ \wt \vecb \cdot \wh F}{  \vecb \cdot \wh F}.$$
\end{deF}

\begin{example}\label{exm:block-length-6}\rm
Let $K$ be a sequence satisfying strong Benford's Law
under $\cF$-expansion, e.g., $\seQ{2^n}$; see 
Theorem \ref{thm:equidistribution-examples}.
Let $\vecb=(1,0,0,0,1,0)$, so $\wt \vecb=(1,0,0,1,0,0)$.
Then, 
$$\prob{\nN : \LB_{6}(K_n) = \vecb}
\ =\ \log_\phi \frac{1+\ome^{3}}{1+\ome^{4}}\approx
0.157.$$
 
\end{example}

\section{Calculations}\label{sec:calculations}
Notice that $\log_{b}(x)$ makes 
it convenient to calculate the distribution of 
the leading digits of exponential sequences $\seQ{a^n}$ under base-$b$ expansion where $b>1$
is an integer.
In this section, we introduce an analogue of 
$\log_b(x)$ for \zec\ expansion in Section 
\ref{sec:continuation}, and use it for various calculations.

As mentioned in the introduction, 
 these functions are merely a tool
 for calculating the leading digits, and 
 in Section \ref{sec:Other-continuations},
  we consider other continuations, and demonstrate  
  their connections to
  	different distributions of leading digits.

 \subsection{An analytic continuation  of the \fib\ sequence}
 \label{sec:continuation}

Below we introduce an analytic continuation of the \fib\ sequence.
 
 \begin{deF}\rm \label{def:alpha}
Let $\al=\frac{\phi}{\sqrt 5}$, and define
$\fF : \real \to \real $ be the function given by 
 $$\mathfrak F(x) =\al(\phi^x+\phi^{-x}\cos(\pi x)\phi^{-2})
.$$
We call the function {\it a Benford continuation of the 
\fib\ sequence}.
\end{deF}
\noindent
Notice that 
$F_n =\frac{1}{\sqrt 5}(\phi^{n+1} -(-1/\phi)^{n+1}) =\frac{\phi}{\sqrt 5}(\phi^n +(-1)^{n}\phi^{-(n+2)})$.
 Thus, $\fF$ 
 is a real   analytic continuation of $F_n$, so 
 $\mathfrak F(n)=F_n$ for all $\nN$.
 It is  an increasing function on $[1,\infty)$.
Let $\mathfrak F\Inv$ denote the inverse function of $\mathfrak F : [1,\infty) \to \real$.
Comparing it with the case of base-$10$ expansion, 
we find that $10^{x-1}$ is an analytic continuation of 
the sequence $\seQ{10^{n-1}}$, and 
its inverse is $1+\log_{10}(x)$, which 
is the main object for the equidistribution for Benford's Law
under base-$10$ expansion.
The equidistribution property described in Theorem 
\ref{thm:equidistr-BL} is associated with strong
\benlaw\ under $\cF$-expansion, 
and   the name of the function is due to this connection.

\begin{lemma}\label{lem:inverse}
For real numbers $x\ge 1$, we have $ \fF(x)=\al \phi^x +O(\phi^{-x})$,
and 
$$\fF\Inv(x)\ =\  \log_\phi (x) - \log_\phi(\al)  + O(1/x^2).$$
\end{lemma}

\begin{proof}
Let $y
 =\al \phi^x +\al\phi^{-x}\cos(\pi x)\phi^{-2}$ and 
 let $w=\al\phi^{-x} \cos(\pi x) \phi^{-2}=O(\phi^{-x})$.
Since $y = \al \phi^x + o(1)$, we have 
 $w = O(1/y)$.
 Then, $ y
 =\al \phi^x +w$ implies
 \AAA{
x &\ =\  \log_\phi(y - w) - \log_\phi\al
\ =\ \log_\phi(y) - \log_\phi\al + \log_\phi(1-w/y)\\
&\ =\ \log_\phi(y) - \log_\phi\al + O(\abs w/y) 
\ =\ \log_\phi(y) - \log_\phi\al + O(1/y^2).
  }

\end{proof}

\subsection{Equidistribution}
Recall the set $\cF_s$ of leading blocks from Definition \ref{def:B-tilde}.
In this section,  having a leading block $\vecb\in\cF_s$ is interpreted in terms of 
the fractional part of the values of  $\fF\Inv$.
\begin{deF}\rm \label{def:vert}
Given $\ep\in  \nat_0^t$ and an integer $s\le t$,
let $\ep\vert s:=(\ep(1),\dots, \ep(s))$.
\end{deF}
Recall $\wh F$ from Definition \ref{def:ome} and the product notation
from Definition \ref{def:products}.
\begin{lemma}\label{lem:LB-inequality-2}
Let $K$ be a sequence of positive real numbers approaching $\infty$, and let $s$ be an integer $\ge 2$.
Let $\vecb\in\cF_s$, and
  let  $A_\vecb:=\set{\nN : \LB_{s}(K_n)=\vecb }$.
Then, there are real numbers $\gamma_n=o(1)$ and $\wt \gamma_n=o(1)$
such that 
  $n\in A_\vecb$ if and only if
\begin{equation}
\log_\phi
  	  \vecb \cdot \wh F +\gamma_n\ \le\ \frc{\fF\Inv(K_n)}\ <\ \log_\phi
  	  \wt \vecb\cdot \wh F + \wt\gamma_n \label{eq:BL-inequality}
\end{equation}  
  	  where $ \wt\gamma_n=0$ if $\vecb$ is the largest block of length $s$.
\end{lemma}

\begin{proof}
Suppose that $\nN$ is sufficiently large, so that $\vecb':=\LB_s(K_n)$ exists.
By \zec's Theorem, there is $\mu\in\cF$ such that 
$K_n=\mu*F$, so $m:=\len(\mu) \ge s$, and  $\vecb'=\mu\vert s $.
There are $\ep\in\cF$ of length $m$ and a \cf\ $\check\ep$ of length $m$ such that $\ep\vert s=\vecb'$, 
$\check\ep\vert s =\wt \vecb'$, $\ep(k)=\check\ep(k)=0$ for all $k>s$,
so $\ep*F \le K_n <\check \ep *F$.
Recall $\al$ from Definition \ref{def:alpha}.
Then,
\GGG{
 \ep * F \ =\ \al \sum_{k=1}^s\ep(k)\phi^{m-k+1}+O(1) 
\ =\  \al \phi^m(1+o(1))\sum_{k=1}^s\ep(k)\ome^{k-1}
\ =\  \al \phi^m(1+o(1))\,\vecb'\cdot \wh F. \\
 \intertext{By Lemma \ref{lem:inverse},}
 \fF\Inv(\ep*F)\ =\  m+\log_\phi (\vecb'\cdot \wh F) +\gamma_n,\quad \gamma_n\ =\ o(1).
} 
Similarly, we have $\fF\Inv(\check\ep*F)= m+\log_\phi (\wt \vecb'\cdot \wh F) +\wt\gamma_n$
where $\wt\gamma_n=o(1)$.
If $\vecb'$ is the largest block of length $s$, then $\check\ep * F = F_{m+1}$, and hence,
$\fF\Inv(\check\ep*F)=m+1$, which implies $\wt\gamma_n=0$.
In general, $\check \ep * F\le F_{m+1}$, so $\fF\Inv(\check\ep * F)\le m+1$.

Thus, if $n\in A_\vecb$, then   $\vecb'=\vecb$, and 
\GGG{
\ep * F \le K_n \ <\ \check \ep * F 
\implies
\fF\Inv(\ep * F) \le\fF\Inv( K_n )\ <\ \fF\Inv( \check \ep * F)\\
\implies
\log_\phi
  	  \vecb \cdot \wh F +\gamma_n \ \le\ \frc{\fF\Inv(K_n)}\ <\  \log_\phi
  	  \wt \vecb\cdot \wh F + \wt\gamma_n.
}
There is no difficulty in reversing this argument, and we leave  the proof of the converse to the reader.    
\end{proof}

\begin{theorem}\label{thm:equidistr-BL}
Let $K $ be an increasing  sequence of positive integers such that 
 $\frc{\mathfrak F\Inv(K_n)}$  is equidistributed.
Then, $K$ satisfies strong Benford's Law under the $\cE$-expansion.  
\end{theorem}

\begin{proof} 
Notice that $\prob{\nN : \LB_s(K_n)=\vecb}$ where $s\ge 2$ is 
equal to the probability of $n$ satisfying (\ref{eq:BL-inequality}).
Let $t\in\nat$.  Then, there is an integer $M_t$ such that $\abs{\gamma_n}$
and $\abs{\wt\gamma_n}$ are $<1/t$ for all $n\ge M_t$.
Thus, by Lemma \ref{lem:LB-inequality-2},
\GGG{
\prob{k \in \Omega_n :  \LB_s(K_n)=B} +o(1)\HSW{.5}\\
\le\  \proB{k \in \Omega_n : \log_\phi
  	  \vecb \cdot \wh F-\tfrac1t\ \le\ \frc{\fF\Inv(K_n)}\ <\ \log_\phi
  	  \wt \vecb\cdot \wh F + \tfrac1t  }+o(1)\\
  	  \implies\ 
\limsup_n \prob{k \in \Omega_n :  \LB_s(K_n)=\vecb}
\ \le\  \log_\phi\frac{\wt \vecb\cdot \wh F}{  \vecb\cdot \wh F} +\frac2t.\\
\prob{k \in \Omega_n :  \LB_s(K_n)=\vecb} +o(1)\HSW{.5}\\
\ge\ \prob{k \in \Omega_n : \log_\phi
  	  \vecb \cdot \wh F+\tfrac1t \le\frc{\fF\Inv(K_n)}\ <\  \log_\phi
  	  \wt \vecb\cdot \wh F - \tfrac1t  }+o(1)\\
  	  \implies\  
\liminf_n\ \prob{k \in \Omega_n :  \LB_s(K_n)=\vecb}
\ \ge\ \log_\phi\frac{\wt \vecb\cdot \wh F}{  \vecb\cdot \wh F}-\frac2t.
  	  }
Since $\liminf$ and $\limsup$ are independent of $t$,
we prove that $\prob{\nN : \LB_s(K_n)=\vecb}=\log_\phi\frac{\wt \vecb\cdot \wh F}{  \vecb\cdot \wh F}$.

\end{proof}
\noindent
The converse of Theorem \ref{thm:equidistr-BL} is true as well, i.e.,
if $K$ satisfies
strong \benlaw\ under $\cF$-expansion, then
$\frc{\fF\Inv(K_n)}$ is equidistributed.
We shall prove it in Section \ref{sec:Other-continuations}. 

The following lemma is useful, and it is probably known.
\begin{lemma}\label{lem:o(1)}
Let $h : \nat \to \real$ be a function such that $\frc{h(n)}$ is equidistributed, and 
let $E : \nat \to \real$ be a function such that $E(n)\to 0$ 
as $n\to \infty$.
Then, $\frc{ h(n) + E(n) }$ is equidistributed.
\end{lemma} 

\begin{cor}\label{cor:equidistribution-examples}
Let $K$ be a sequence of positive integers given by
 $K_n= a b^n(1+o(1)) $ where $a$ and $b$ are positive real numbers such that $\log_\phi b$ is irrational.
Then,   $\frc{\fF\Inv( K_n )}$ is equidistributed, and hence,
given $\vecb\in\cF_s$ where $s\ge 2$,
$$\prob{\nN : \LB_{s}(K_n) = \vecb} 
\ =\  \log_\phi\frac{ \wt \vecb \cdot \wh F}{  \vecb \cdot \wh F}.$$
\end{cor}

\begin{proof} 
By Lemma \ref{lem:inverse},
\GGG{
\fF\Inv(K_n) \ =\ n\log_{\phi}b 
-\log_\phi(a/ \al) + \log_{\phi} (1+o(1)) +o(1).
}
Since $\log_\phi b$ is irrational, 
by Weyl's Equidistribution Theorem,
$\frc{n\log_\phi b}$ is equidistributed, and by the lemma,
$\frc{n\log_\phi b  + o(1) }$ is equidistributed.
Shifting it by a constant $-\log_\phi(a/\al)$ does not change the 
equidistribution property, and this concludes the proof.
\end{proof}
For example,
if $K$ is a sequence given by $K_n=  \sum_{k=1}^N a_k b_k^n $ where $a_k,b_k\in \zz$, $a_1>0$, 
 and $b_1>\abs{b_k} $ for all $k\ge 2$,
 then $K_n= a_1 b_1^n(1+o(1))$, and 
 $\frc{\fF\Inv(K_n)}$ is equidistributed.
Many increasing  sequences $K$ of positive integers given by a linear recurrence with constant positive integer coefficients satisfy 
  $K_n =  a b^n(1+o(1))$
 where $\log_\phi(b)$ is irrational, and hence,
$\frc{\fF\Inv(K_n)}$ is equidistributed.

\subsection{The leading blocks of integer powers}

Let $a$ be a positive integer,
and let $K$ be the sequence given by $K_n= n^a$.
Then,  $K$ does not satisfy Benford's Law under 
the base-10 expansion, but it has a close relationship with Benford's Law \cite{hurlimann}.
In this section, we show that both statements are true under $\cE$-expansion as well.
Recall $\Omega_n$ from Notation \ref{def:N0} and $\cF_3$ from Definition \ref{def:F-blocks}, and let $\vecb_1:=(1,0,0)\in \cF_3$.
We also introduce the oscillating behavior of
$\prob{k \in \Omega_n: \LB_3(K_k)=\vecb_1 }$ as $n\to\infty$, and hence,
$\prob{\nN : \LB_3(K_n)=\vecb_1}$ does not exist.

\begin{example}\rm
Let $K$ be the sequence given by $K_n=n$, and let $t>0$ be a large integer. Given a sufficiently large positive random integer 
$n<F_{t+1} $, let $n=\mu* F$ be the $\cF$-expansion, and $M:=\len(\mu) $. 
Notice that 
  $\LB_3(n)=\vecb_1$ if and only if $n= F_M + m$ where 
  $0\le m  <F_{M-2}$.  Thus,
  there are $F_{M-2} $ integers $n$ in $[1, F_{t+1})$
 such that $F_M\le n <F_{M+1}$ and $\LB_3(n)=\vecb_1$.
  Thus,
\GGG{
\prob{ n\in \Omega_{F_{t+1}} :  \LB_3(n)=\vecb_1}
\ =\ \left( \frac1{F_{t+1} } 
\sum_{M=3}^t F_{M-2}  \right) + o(1)  
=\left( \frac1{F_{t+1} } 
\sum_{M=3}^t \al \phi^{M-2} + o(1) \right) + o(1)  \\
\ =\ \frac1{\al\phi^{t+1}+o(1) } \frac{\al\phi^{ t-1}}{\phi -1} +o(1)
\ =\ \frac1{\phi^2(\phi-1)}+ o(1)\ =\ \phi-1 + o(1)
} 
as function of $t$.
 However, by Theorem \ref{thm:limsupinf},
 we have
 \GGG{
\lim\sup_n \prob{ k \in \Omega_n : \LB_3(k)=\vecb_1}
\ =\  \frac{ \phi + 1}{ \phi+2}\approx .724,\\
\lim\inf_n \prob{ k \in \Omega_n : \LB_3(k)=\vecb_1}
\ =\  \phi - 1 \approx .618.
} 
Thus, $\prob{\nN : \LB_3(n)=\vecb_1}$ does not exist.

\end{example}

Recall    $\fF$ from Definition \ref{def:alpha}, and its inverse $\fF\Inv$.
We use the function $\fF$ to more generally handle the distribution of
the leading blocks of $\seQ{ n^a}$ with any length.
Given a positive integer $m$, let 
$A_m=\set{\nN : n < F_m^{1/a}}$. 
\begin{lemma}\label{lem:Am}
If $\beta\in[0,1]$, then 
\GGG{
\prob{ n \in A_m : \frc{ \fF\Inv(n^a) }\le \beta}
\ =\  \frac{\phi^{\beta/a}- 1}{\phi^{1/a} - 1} + O(m\phi^{-m/a}).
}
\end{lemma}

\begin{proof}
Let $m\in\nat$, and let $n\in A_{m+1}':=A_{m+1} - A_m$, so that 
$F_m \le n^a < F_{m+1}$ and 
$m \le \fF\Inv(n^a) < m+1$.
Thus, given a real number $\beta\in [0,1]$, 
\GGG{
\begin{align*}
\Set{ n \in A_{m+1}':  \frc{\fF\Inv(n^a)} \le \beta  } 
&\ = \ \Set{n \in A_{m+1}' :  m \le \fF\Inv(n^a) \le m+  \beta  }\\
&\ = \ \Set{ n \in A_{m+1}' : 
 \fF(m)^{1/a} \le  n   \le \fF(m+  \beta)^{1/a} }  \\
\implies
\# \Set{ n \in A_{m+1}':  \frc{\fF\Inv(n^a)} \le \beta  } 
&\ = \   \fF(m+\beta)^{1/a} - \fF(m )^{1/a}+O(1)  \\
&
\ = \  \al^{1/a} \phi^{ (m+\beta)/a} -\al^{1/a} \phi^{ m/a} 
  + O(1)
\end{align*}
\\
\implies
\#\Set{ n \in A_{m+1} :  \frc{\fF\Inv(n^a)} \le \beta  } 
\ = \ 
\sum_{k=1}^m
 \al^{1/a} \phi^{ (k+\beta)/a} -\al^{1/a} \phi^{k/a} 
 + O(1)\\
\ = \  \al^{1/a} \phi^{ (m+\beta)/a}\gamma -\al^{1/a} \phi^{m/a} \gamma
 + O(m),\quad
\gamma=\frac{\phi^{1/a}}{\phi^{1/a}-1}.\\
\intertext{This proves that}
\begin{aligned}
\prob{ n \in A_{m+1} :  \frc{\fF\Inv(n^a)} \le \beta  } 
&\ = \  \frac{\al^{1/a} \phi^{ (m+\beta)/a}\gamma -\al^{1/a} \phi^{m/a} \gamma
+   O(m)}{F_{m+1}^{1/a}+O(1)}\\ 
&\ = \   \frac{  \phi^{  \beta/a}\gamma - \gamma
+   O(m\phi^{-m/a})}{ \phi^{1/a}
+O(\phi^{-m/a})} 
\ = \   \frac{  \phi^{  \beta/a}-1}{ \phi^{1/a}-1}
+   O(m\phi^{-m/a}).
\end{aligned} 
}

\end{proof}

Recall from Lemma \ref{lem:LB-inequality-2} that 
$$\prob{ n \in A_m : \LB_3(n^a) = \vecb_1} 
= \prob{ n \in A_m :\frc{\fF\Inv(n^a)} \le \delta_1+o(1)}$$
where $\delta_1:=\log_\phi\frac{ \wt \vecb_1\cdot \wh F}{\vecb_1\cdot \wh F}$.
Thus, as $m\to \infty$,  by Lemma \ref{lem:Am},
\GGG{
\prob { n \in A_m : \LB_3(n^a) = \vecb_1}  
\ \to\  \frac{  \phi^{  \delta_1/a}-1}{ \phi^{1/a}-1}
\ = \ \frac{(1+\ome^2)^{1/a} - 1}{ \phi^{1/a}-1} 
}
where $\ome=\phi^{-1}$. 
Let us show that 
 \GGG{
 \prob{n \in A_m : \LB_3(n^a)=\vecb_1} 
\ \not \to\ \delta_1 
 }
 as $m\to\infty$.
We claim that the ratio $\frac{(1+\ome^2)^{1/a} - 1}{ \phi^{1/a}-1}$ is not equal to $\delta_1=\log_\phi(1+  \ome^2 )$.
Since $a\in\nat$, the ratio is an algebraic number over $\ratn$.
However, by the Gelfand-Schneider Theorem, 
$\log_\phi(1+ \ome^2)$ is a transcendental number.
Thus, $K$ does not satisfy Benford's Law under 
the $\cE$-expansion. 

However, as noted in \cite{hurlimann} for base-$b$ expansions,
we have
\GGG{
\lim_{a\to\infty}\lim_{m\to\infty}
\prob { n \in A_m : \LB_3(n^a) = \vecb_1}
\ = \ \lim_{a\to\infty} 
 \frac{  \phi^{  \delta_1/a}-1}{ \phi^{1/a}-1}  
\ = \  \delta_1\ = \  \log_\phi(1+  \ome^2 ).
}
Even though the leading blocks of $K_n$ do not satisfy Benford's Law
under $\cE$-expansion,
  the limiting behavior of high power sequences for special values of $n$ 
  resembles
  Benford's Law.

Recall $\Omega_n$ from Definition \ref{def:N0}.
Let us use Lemma \ref{lem:Am} 
to prove that $ \prob{ k \in \Omega_n: \frc{\fF\Inv(K_k)}\le \beta}$ oscillates, and does not converge.
\begin{theorem}\label{thm:limsupinf}
Let $\beta$ be a real number in $ [0,1]$,
and let   $r:=(\phi^{\beta/a} -1)/(\phi^{1/a}-1)$.
Given an integer $n>1$,
let $\fF\Inv(n^a)=m+p$ where  
$p=\frc{\fF\Inv(n^a)}$ and $m\in\nat$. Then, 
\GGG{
P_n:=\prob{ k \in \Omega_n: \frc{\fF\Inv(K_k)}\le \beta}\ = \ 
\begin{cases}
\frac{ r + \phi^{p/a}  -1}
{\phi^{p/a}  }+  O(m\phi^{-m/a} ) 
	&\text{ if $0\le p\le \beta$}\\
	\frac{ r + \phi^{\beta/a}  -1}
{\phi^{p/a}  }+  O(m\phi^{-m/a} )
	&\text{ if $ \beta< p<1$}
\end{cases}.
}
In particular,
\GGG{
\lim\sup P_n  
= r\phi^{1/a- \beta/a}  = \beta + O(1/a),\quad\text{and}\quad
\lim\inf P_n =r= \beta+ O(1/a).
} 

\end{theorem} 

\begin{proof}
Let $m$ be a sufficiently large positive integer,  and let  $n\in A_{m+1}-A_m$.
  Let $n=\fF(m+p)^{1/a}$ for $p\in [0,1)$.
  If $p\le \beta$, then, $\frc{\fF\Inv(n^a)}=\frc{\fF\Inv \fF(m+p)}
=\frc{m+p}=p\le \beta$, and  if $p> \beta$, then, $\frc{\fF\Inv(n^a) }=p> \beta$.
Thus, 
\GGG{
\Set{n\in A_{m+1}-A_m :  \frc{\fF\Inv(n^a)}\le \beta}
\ = \ \Set{n\in A_{m+1}-A_m : n\le \fF(m+\beta)^{1/a}}.\\ 
\intertext{
If $n\le \fF(m+\beta)^{1/a}$, i.e., $p\le \beta$, then by Lemma \ref{lem:Am}}
\begin{aligned}
P_n &\ = \ \frac1n\left(
\prob{ k\in A_m : \frc{\fF\Inv(k^a)}\le \beta }
\ \#A_m +  n - \fF(m)^{1/a} +O(1) \right)\\
& \ = \ \frac{ r \fF(m)^{1/a} + O(m )+  \fF(m+p)^{1/a}  - \fF(m)^{1/a}}
{\fF(m+p)^{1/a}+O(1)}\\
&\ = \ \frac{ r + O(m\phi^{-m/a} )+ \phi^{p/a}  -1}
{\phi^{p/a} +O(\phi^{-m/a})}
\ = \ \frac{ r + \phi^{p/a}  -1}
{\phi^{p/a}  }+  O(m\phi^{-m/a} )
\end{aligned}\\
\intertext{If $n> \fF(m+\beta)^{1/a}$, i.e., $p>\beta$, then
}
P_n
= \frac{ r + \phi^{\beta/a}  -1}
{\phi^{p/a}  }+  O(m\phi^{-m/a} )
\ = \ \frac{ r\phi^{1/a} }
{\phi^{p/a}  }+  O(m\phi^{-m/a} ).\\
\text{Thus, }
\lim\sup P_n = \frac{ r + \phi^{\beta/a}  -1}{\phi^{\beta/a}  }
= \frac{ r\phi^{1/a} }
{\phi^{\beta/a}  },\quad
\lim\inf P_n   
= \frac{ r\phi^{1/a} }
{\phi^{1/a}  }=r.
 }
 \end{proof}
Thus, $\prob{\nN : \frc{\fF\Inv(K_n)}\le \beta}$ does not converge,
but $\frc{\fF\Inv(K_n)}$ is almost equidistributed for large values of $a$.

\begin{example}\rm
Let $\vecb$ and $\wt \vecb$ be the blocks defined in Example \ref{exm:block-length-6}, and let $K$ be the 
sequence given by $K_n=n^2$.
By   Lemma \ref{lem:LB-inequality-2}, 
if $D:=\set{ \nN : \LB_6(K_n)=\vecb }$,
then for $n \in D$, 
 \GGG{
  \log_\phi (1+\ome^{4}) + o(1) 
 \ <\ \frc{\fF\Inv(K_n) }\ <\ \log_\phi(1+\ome^{3}) + o(1) 
  }
  where the upper and lower bounds are
  functions of $n\in D$.
  Let $\beta=\log_\phi (1+\ome^{4})$ and 
  $\wt \beta = \log_\phi(1+\ome^{3})$.
  Recall $\Omega_n$ from Definition \ref{def:N0}.
  Then,  
  \GGG{
  \prob{k\in \Omega_n : \LB_6(K_k)=\vecb }=\HSW{.6}\\
  \prob{k\in \Omega_n : \frc{\fF\Inv(K_n)} <\wt \beta }
 \ -\ \prob{k\in \Omega_n : \frc{\fF\Inv(K_n)}<\beta}\  +\ o(1).\\
 \intertext{ Let $r=(\phi^{\beta/2}-1)/(\phi^{1/2}-1)$ 
  and $\wt r=(\phi^{\wt \beta/2}-1)/(\phi^{1/2}-1)$, 
  and let  $n=\fF(m+p)^{1/a}$
  where $p=\frc{\fF\Inv(n^a)}\in [0,1)$.
  Then, by Theorem \ref{thm:limsupinf}, we have}
 \prob{k\in \Omega_n : \LB_6(K_k)=\vecb }\ = \ 
 \begin{cases}
 \frac{\wt r + \phi^{p/2}-1 }{\phi^{p/2}}
 - \frac{  r + \phi^{p/2}-1 }{\phi^{p/2}} +o(1) 
 &\text{ if $p\le \beta$ },\\
  \frac{\wt r + \phi^{p/2}-1 }{\phi^{p/2}}
 - \frac{  r + \phi^{\beta /2}-1 }{\phi^{p/2}}+o(1)
 &\text{ if $\beta < p \le \wt \beta$ }\\
  \frac{\wt r + \phi^{\wt \beta/2}-1 }{\phi^{p/2}}
 - \frac{  r + \phi^{\beta /2}-1 }{\phi^{p/2}}+o(1)
 &\text{ if $p>\wt \beta$ }.
\end{cases} \\
\begin{align*}
\implies\limsup_n \prob{k\in \Omega_n : \LB_6(K_k)=\vecb }
&\ = \ \frac{\wt r + \phi^{\wt \beta/2}-1 }{\phi^{\wt \beta/2}}
 - \frac{  r + \phi^{\beta /2}-1 }{\phi^{\wt \beta/2}}\approx 0.1737\\
\liminf_n \prob{k\in \Omega_n : \LB_6(K_k)=\vecb }
&\ = \ \frac{\wt r + \phi^{\beta/2}-1 }{\phi^{\beta/2}}
 - \frac{  r + \phi^{\beta/2}-1 }{\phi^{\beta/2}}\approx 0.1419.
\end{align*} 
 }
  
\end{example}

\section{Other continuations}\label{sec:Other-continuations}
 
Reflecting upon  Lemma \ref{lem:LB-inequality-2} and Theorem \ref{thm:equidistr-BL}, 
we realized that we could consider  different continuations of  
 the \fib\ sequence $F$, and ask which sequence
satisfies the equidistribution property, and which distributions
its leading blocks follow.
Let us demonstrate the idea in Example \ref{exm:zec-uniform}.
The claims in this example can be proved
  using Theorem \ref{thm:continuation}.
  Recall the Benford continuation $\fF$ from Definition \ref{def:alpha}.
\begin{deF}\rm \label{def:fFn}
Given $\nN$, let 
$\fF_n : [0,1] \to [0,1]$  be the increasing function given by 
 \GGG{ 
 \fF_n(p)\ :=\  \frac{\fF(n+p) - \fF(n)}{\fF(n+1) - \fF(n)}
 =\frac{\fF(n+p) - \fF(n)}{F_{n-1}}\ = \ 
 \phi (\phi^p -1) + o(1) 
 }
where $F_0:=1$.
Let $\fF_\infty: [0,1] \to [0,1]$  be the increasing function given by 
$\fF_\infty(p)=\phi(\phi^p-1)$.
\end{deF}

Recall uniform continuations of sequences from Definition \ref{def:uniform-continuation}.
\begin{lemma}\label{lem:fF-uniform}
The function $\fF$ is a uniform continuation of $F$.
\end{lemma}

\begin{proof}
Notice that $\fF_n(p) = \phi(\phi^p-1) + \gamma(n,p)$ where 
$\abs{\gamma(n,p)}<C/\phi^n$ where $C$ is independent of $p$ and $n$.
Thus, it uniformly converges to $ \phi(\phi^p-1) $.
\end{proof}

\begin{lemma}\label{lem:Kn}
Let $p\in[0,1]$ be a real number. 
Then, $\fF(n + \fF_n\Inv(p ) )=F_n + (F_{n+1} - F_n) p$.
\end{lemma}
\begin{proof}
Let $p'=\fF_n\Inv(p )$.
Then, $\fF_n(p')= p$, and hence, $\frac{ \fF(n+p')-\fF(n)}{F_{n+1} - F_n}=p$.
The assertion follows from the last equality.
\end{proof}

\begin{example}\label{exm:zec-uniform}
\rm
  Let   $f : [1,\infty) \to \real$ be the increasing continuous function  whose 
graph is the union of the line segments from $(n,F_n)$ to $(n+1,F_{n+1})$
for $\nN$.  Then, $f_\infty(p)=p$ for all $p\in[0,1]$.
Let $K$ be the sequence given by $K_n= \flr{\fF(n + \fF_n\Inv( \frc{n\pi} )) } $.
Then, by Lemma \ref{lem:Kn},
$$ f\Inv\left(\fF(n + \fF_n\Inv( \frc{n\pi} ) \right) = n  + \frc{n \pi}
\ \implies\  \frc{f\Inv(K_n)} =\frc{n \pi}+o(1),$$
which is equidistributed.

Recall   $\cF_s$ from Definition \ref{def:B-tilde} where $s\ge 2$, and let $\vecb\in\cF_s $.
Recall $\wh F$ from Definition \ref{def:ome} and the product notation from Definition \ref{def:products}.
Then, 
by Theorem \ref{thm:continuation}, 
\GGG{
 \prob{\nN :  \LB_{s}(K_n)=\vecb }\ =\ %
 \phi(\wt \vecb\cdot \wh F - \vecb\cdot \wh F)
 \ =\    \phi^{-s+2} (\wt \vecb* \baR F- \vecb*  \baR F)
 }
  where $\baR F$ is the sequence given by $\baR F_n= \phi^{n-1}   $.
If $\vecb(s)=0$, then $  \ome^{s-2} (\wt \vecb* \baR F- \vecb*  \baR F)
=\ome^{s-2} $, and 
if $\vecb(s)=1$, then  $ \ome^{s-2} (\wt \vecb* \baR F- \vecb*  \baR F)
=\ome^{s-1} $.
For example, if $s=6$, then 
\begin{align*} 
  \prob{\nN :  \LB_{6}(K_n)\ = \ (1, 0, 0, 1,0,1)} &\ =\ \ome^{5}\\
  \prob{\nN :  \LB_{6}(K_n)\ = \ (1, 0, 1, 0,1,0)}
 &\ =\  \ome^{4}.
\end{align*} 
It's nearly a uniform distribution. 

Let us show that the probabilities   add up to $1$.
Notice that 
$\#\cF_s= F_{s-1}$, 
$\#\set{ \vecb\in \cF_s : \vecb(s)=0}= F_{s-2}$, and 
and $\#\set{ \vecb\in \cF_s : \vecb(s)=1}= F_{s-3}$.
Then,  
by Binet's Formula, the following sum is equal to $1$:
\GGG{
 \sum_{\vecb \in \cF_s}  \ome^{s-2} (\wt \vecb* \baR F- \vecb*  \baR F) \ = \ 
\frac{F_{s-2}}{\phi^{s-2}} 
+ \frac{F_{s-3}}{\phi^{s-1}}=1 .
}

By Lemma \ref{lem:Kn}, we have 
$K_n = \flr{F_n + (F_{n+1} - F_n)\, \frc{n \pi}}$ for $\nN$, and the following are 
the first ten values of $ K_n $:
$$(1, 2, 3, 6, 11, 19, 33, 36, 64, 111).$$
\end{example}

 Let us introduce and prove the main results on continuations.

\begin{lemma}\label{lem:uniform-continuation}
Let $f$ be a uniform continuation of $F$, and 
let $K$ be a sequence
  of positive real numbers approaching $\infty$.
Then, $\frc{ f\Inv(\flr{K_n}) }=\frc{ f\Inv(K_n) }+o(1)$.
\end{lemma}
\begin{proof}
Let $\nN$.  Then, $F_m\le \flr{K_n}\le K_n < F_{m+1}$ for $m\in\nat$ depending on $n$.
Let $K_n = f(m+p)$ and $\flr{K_n}=f(m+p')$ where $p,p'\in[0,1]$ are real numbers, which depend on $n$.
Then, $F_m + f_m(p')(F_{m+1}-F_m) + O(1) = F_m + f_m(p)(F_{m+1}-F_m) $, and hence,
$f_m(p')+o(1) = f_m(p)$.
Thus,
\GGG{
f\Inv(K_n) \ =\  m +p =m + f_m\Inv\left(f_m(p')+o(1)\right)\ =\ m + f_m\Inv\left(f_\infty(p')+o(1)\right).\\
\intertext{By the uniform convergence,}
 \ =\  m + f_\infty\Inv\left(f_\infty(p')+o(1)\right) + o(1)\ =\  m + f_\infty\Inv\left(f_\infty(p') \right) + o(1) 
 \ =\  m+p' +o(1).
}
Therefore, $\frc{f\Inv(K_n) }= \frc{f\Inv(\flr{K_n}) }+o(1)$.

\end{proof}

\begin{theorem}\label{thm:continuation}
Let $f : [1,\infty) \to \real$ be  a  uniform continuation of $F$. 
    Then  there is a  sequence
    $K$ of positive integers approaching $\infty$, e.g.,
 $K_n=\flr{\fF\left(n + \fF_n\Inv\circ f_n ( \frc{n\pi} \right)} $, such that  
    $\frc{ f\Inv(K_n) }$ is equidistributed. 

Let $K$ be a sequence of of positive integers approaching $\infty$ such that  
    $\frc{ f\Inv(K_n) }$ is equidistributed.
Let $\vecb\in\cF_s$ where $s\ge 2$.
Then, 
\begin{align}
\prob{\nN : \LB_{s}( {K_n}) =\vecb }  
 &\ = \ 
f_\infty\Inv
\circ \fF_\infty(\log_\phi\wt  \vecb\cdot \wh F )  - 
f_\infty\Inv\circ   \fF_\infty(\log_\phi \vecb\cdot \wh F ) \notag
 \\
&\ = \ 
f_\infty\Inv
\left(\phi(\wt  \vecb\cdot \wh F - 1) \right) - 
f_\infty\Inv\left( \phi(\vecb\cdot \wh F - 1) \right).
\label{eq:prob-density}
\end{align}

\end{theorem}

\begin{proof}
Let $x\ge 1$ be a real number, and let 
$F_n \le x <F_{n+1}$ for $\nN$. 
Since $\fF$ and $f$ are increasing continuations of $F$, 
there are two unique real numbers $p$ and $p'$ in $[0,1]$ such that 
$x=\fF(n+p)=f(n+p')$. We claim that 
\begin{equation}
f\Inv(x)=n+ f_n\Inv(\fF_n(p)), \label{eq:f-inverse}
\end{equation} 
and $\fF\Inv(x) =n+\fF_n\Inv(f_n(p'))$.
To prove the claim, note 
\GGG{
\fF(n+p)=f(n+p')
\ \implies\ 
F_n + \fF_n(p) (F_{n+1}-F_n) 
\ = \ F_n + f_n(p') (F_{n+1}-F_n) \\
\ \implies\  p'= f_n\Inv(\fF_n(p)),\ %
p = \fF_n\Inv(f_n(p')).
}
Then  $f(n+p')=x$ and $\fF(n+p)=x$ imply the claim.

Let $ \overline K$ and $K$ be the sequences given by 
$\overline K_n= \fF\left(n + \fF_n\Inv\circ f_n ( \frc{n\pi} ) \right)  $ and
$K_n=\flr{\overline K_n}$.
Given $\nN$,
let
$p_n=\fF_n\Inv\circ f_n( \frc{n\pi} )$.
Then,
\GGG{
f\Inv(  \overline K_n)\ =\  n + f_n\Inv\big(\fF_n(p_n)\big)\ =\ n+\frc{n\pi}.
}
Thus, $\frc{f\Inv(  \overline K_n)}$ is equidistributed.
If we further assume that $f$ is a uniform continuation, then, by 
Lemmas \ref{lem:o(1)} and  \ref{lem:uniform-continuation}, 
$\frc{f\Inv(\flr{ \overline K_n})}=\frc{f\Inv( K_n )} $ is equidistributed as well. 

Let $K$ be a sequence of of positive integers approaching $\infty$ such that  
    $\frc{ f\Inv(K_n) }$ is equidistributed.
Let $\vecb\in\cF_s$, and let  $A_\vecb:=\set{\nN : \LB_{s}( {K_n})=\vecb }$.
Let $n\in A_\vecb$, and 
$F_m \le  {K_n} <F_{m+1}$ for $m\in\nat$ depending on $n$.
Let $ {K_n}=\fF(m+p)=f(m+p')$ where $p$ and $p'$ are real numbers in 
$[0,1]$ depending on $n$.

Then, by Lemma \ref{lem:LB-inequality-2},
\AAA{ 
  	  &\quad \log_\phi
  	  \vecb \cdot \wh F +o(1)\ \ <\  \frc{\fF\Inv( {K_n})}\ <\  \log_\phi
  	  \wt \vecb\cdot \wh F +o(1) \\
  	  \implies &\quad
  \log_\phi
  	  \vecb \cdot \wh F +o(1) \ <\ \frc{m + \fF_n\Inv(f_n(p'))}\ <\  \log_\phi
  	  \wt \vecb\cdot \wh F +o(1)	  
  	  \\
  	  \implies &\quad
  \log_\phi
  	  \vecb \cdot \wh F +o(1) \ <\ \fF_n\Inv(f_n(p')) \ <\  \log_\phi
  	  \wt \vecb\cdot \wh F +o(1)	   \\
  	  \implies &\quad
f_n\Inv\circ \fF_n( \log_\phi
  	  \vecb \cdot \wh F +o(1) ) \ <\ p'\ <\ f_n\Inv\circ \fF_n(  \log_\phi
  	  \wt \vecb\cdot \wh F +o(1)	  )  \\
  	  \implies &\quad
f_\infty\Inv\circ\fF_\infty( \log_\phi
  	  \vecb \cdot \wh F )  +o(1)\ <\ 
  	  \frc{f\Inv(K_n)}\ <\ f_\infty\Inv\circ\fF_\infty(  \log_\phi
  	  \wt \vecb\cdot \wh F )  +o(1).
  	  }
Since $\frc{f\Inv(K_n)}$ is equidistributed,
 the above inequalities imply the assertion (\ref{eq:prob-density}).
\end{proof}

Let us demonstrate a continuation, for which 
the distribution of leading blocks of length $4$ coincides
 with that of strong \benlaw, but 
 the distribution does not coincide for higher length blocks.
\begin{example}\rm \label{exm:fake-BF}
Consider $\cF_4=\set{\vecb_1,\vecb_2,\vecb_3}$, i.e.,
$$
\vecb_1=(1,0,0,0),\ \vecb_2=(1,0,0,1),\ \vecb_3=(1,0,1,0).$$
Let $p_k = \log_\phi(\vecb_k\cdot \wh F)<1$ for $k=1,2,3$, and 
let $p_0=0$ and $p_4=1$.
For each $\nN$, define $f_n : [0,1]\to [0,1]$ to be the function whose graph 
is the union of line segments 
from $(p_k, \fF_\infty(p_k))$ to $(p_{k+1}, \fF_\infty(p_{k+1}))$
for $k=0,1,2,3$.
Notice that $f_n$ is defined independently of $n$, and that 
it defines a uniform continuation $f : [1,\infty) \to [1,\infty)$ such that 
$f_\infty = f_n$ for all $\nN$ as follows:
Given $x\in [1,\infty)$, find $\nN$ such that $n\le x<n+1$, and 
define $f(x) =F_n + f_n(x-n)(F_{n+1} - F_n)$.
 
 Note that $f_\infty(p_k)=\fF_\infty(p_k)$, i.e.,
 $f_\infty\Inv(\fF_\infty(p_k))=p_k$
for $k=0,1,2,3$.
By Theorem \ref{thm:continuation},
if $\frc{f\Inv(K_n)}$ is equidistributed, we have 
 $$\prob{\nN : \LB_4(K_n) =\vecb_k} \ =\ 
 p_{k+1} - p_k
\ =\ 
\log_\phi \frac{\wt  \vecb_{k}\cdot \wh F}{\vecb_k\cdot \wh F} $$
 where $\wt\vecb_3=(1,0,1,1)$ as defined Definition \ref{def:B-tilde}.
However, the leading blocks of length $>4$ do not satisfy \benlaw\ under $\cF$-expansion.
\end{example}

The following is an example where $f_\infty$ is analytic.
\begin{example}\rm

Let  $f : [1,\infty) \to \real$
be the function 
 given by $ f(n+p) = F_n + (F_{n+1}-F_n) p^2$ 
 where $\nN$ and $p\in[0,1)$.
 Then, $f_\infty(p)=p^2$.
 
 Let $K$ be the sequence given by $K_n=\flr{\fF(n + \fF_n \Inv(p^2))}$,
 and let $\vecb\in\cF_s$.
Then, Theorem \ref{thm:continuation}, 
\GGG{
\prob{\nN : \LB_{s}( {K_n}) =\vecb }   
 =
\sqrt{ \phi(\wt  \vecb\cdot \wh F - 1) } - 
\sqrt{ \phi(\vecb\cdot \wh F - 1) }.
}
 
\end{example}

\subsection*{Converse}

Let's consider the converse of Theorem \ref{thm:continuation}, i.e., given a sequence $K$ of positive integers 
approaching $\infty$, 
let us 
construct a uniform continuation $f$, if possible, such that $\frc{f\Inv(K_n)}$ is equidistributed.
Recall the set $\cF_s$ from Definition \ref{def:B-tilde}.

\begin{deF}\rm 
A sequence $K$ of positive integers approaching $\infty$ is said to have 
{\it  strong leading block distribution under $\cF$-expansion} if  
 $\prob{\nN : \LB_{s}(K_n)=\vecb}$ exists for each integer $s \ge 2$ and each $\vecb\in\cF_s$.

\end{deF}
 
 \begin{example}\label{exm:lucas}\rm 
Let  $K$ be the Lucas sequence, i.e.,
$K=(2,1,3,4,\dots)$ and $K_{n+2}=K_{n+1}+K_n$.
Recall that $F_n = \frac1{10}(5+\sqrt 5) \phi^n(1+o(1))$
and $K_n  = \frac12( \sqrt 5-1) \phi^n(1+o(1))$, and 
let  $\al=\frac1{10}(5+\sqrt 5) $ and $a= \frac12( \sqrt 5-1) $.
Then,  by Lemma \ref{lem:inverse},
$$
 \frc{\fF\Inv(K_n)}=-\log_\phi(a/\al) + o(1) 
 \approx .328  + o(1). 
 $$  
By Lemma \ref{lem:LB-inequality-2}, the leading block of $K_n$ being $\vecb_1=(1,0,0)$
is determined by whether $0\le \frc{\fF\Inv(K_n)}  <
 \log_\phi(1+\ome^{2}) \approx .67$.
Thus,  
$\prob{\nN : \LB_3(K_n)=\vecb_1}=1$, and 
$\prob{\nN : \LB_3(K_n)=\vecb_2}=0$.

In fact, the sequence $K$  has  strong 
\lbd.
Recall $\wh F$ from Definition \ref{def:ome}, and let us claim  that  $\vecb\cdot \wh F \ne \frac \al a=\frac1{10}(5 + 3\sqrt 5)$ for 
all $s\in\nat$ and $\vecb\in \cF_s$.
Notice that 
\begin{equation}
\frac\al a - 1 = \sum_{k=1}^\infty \ome^{4k}. \label{eq:lucas}
\end{equation} 
The equality (\ref{eq:lucas}) is called {\it the  \zec\ expansion of a real number in $(0,1)$} since
it is a power series expansion in $\ome$ where no consecutive powers are used; a formal definition is given in   Definition \ref{def:infinite-tuples} below.
By the uniqueness of \zec\ expansions of the real numbers 
in $(0,1)$, the above infinite sum in (\ref{eq:lucas}) is not equal to any finite sum
$\vecb\cdot \wh F-1$ where $\vecb\in\cF_s$; see Theorem \ref{thm:ZT-OI}. 

Let $s$ be an integer $\ge 2$, and let $\cF_s=\set{\vecb_1,\dots,\vecb_\ell}$.
Then, there is $k\in \nat$ such that 
  $\vecb_k \cdot \wh F\ <\  \frac\al a  \ <\ \vecb_{k+1} \cdot \wh F$.
This implies that 
$$
\log_\phi(\vecb_k \cdot \wh F) \ <\ \log_\phi( \tfrac\al a )\ <\ \log_\phi(\vecb_{k+1} \cdot \wh F).
$$
  Since $\frc{\fF\Inv(K_n)}= \log_\phi(\al/a) + o(1)$ for all 
  $\nN$,
  by Lemma \ref{lem:LB-inequality-2}, we have
  $\prob{\nN : \LB_s(K_n)=\vecb_k}=1$.
  For example, consider the case of $s=9$, and notice that 
  $\ome^4 +\ome^8 < \frac\al a - 1 < \ome^4 +\ome^7$ by (\ref{eq:lucas}).
  Then, we have 
  $\vecb \cdot \wh F <  \frac\al a < \wt \vecb \cdot \wh F$ where
 $$ \vecb=(1,0,0,0,1,0,0,0,1)\ \text{ and }\ 
 \wt \vecb = (1,0,0,0,1,0,0,1,0),$$
  and the probability of 
  having the leading block $\vecb$ in the values of the Lucas sequence is 
  $1$.

\end{example}
\noindent  
Recall uniform continuations from Definition \ref{def:uniform-continuation}.
Since the distribution of the leading blocks of the Lucas sequence $K$ is concentrated on one particular block in $\cF_s$ for each  $s$,
there does not exist a uniform continuation $f$, described in Theorem 
\ref{thm:continuation}, whose 
equidistribution is associated 
with the leading block distributions of the Lucas sequence $K$.
For a uniform continuation to exist, the values of the leading block distributions 
must be put together into a continuous function, and below
we formulate the requirement more precisely.

\begin{deF}\label{def:infinite-tuples}\rm
Let $\OI$ denote the interval $(0,1)$ of real numbers.
An infinite tuple $ \mu\in\prod_{k=1}^\infty \nat_0$ is 
called a {\it \zec\ expression for $\OI$} if
$ \mu(k)\le 1$, $ \mu(k)  \mu(k+1)=0$, and  for all $j\in\nat_0$,
the sequence 
$\seQ{ \mu(j+n)}$ is not equal to the sequence $ \seQ{1+(-1)^{n+1})/2}=(1,0,1,0,\dots)$.
Let $\cFst$ be the set of \zec\ expressions for $\OI$.

Given  $s\in\nat$ and $ \mu\in \cFst$, let $ \mu\vert s :=( \mu(1),\dots, \mu(s)) $.
Given $s\in\nat$ and  $\set{ \mu,\tau}\subset\cFst$, we declare $ \mu\vert s<\tau\vert s$ if $ \mu\vert s \cdot \wh F
< \tau\vert s\cdot\wh F$, which coincides with the \lex\ order on $\cF$. 
\end{deF}

\begin{notation}
\rm \label{notation:dot-product}
Given a sequence $Q$ of real numbers, and $\mu\in  \prod_{k=1}^\infty \nat_0$, 
we define $ \mu\cdot Q:=\sum_{k=1}^\infty \mu(k)Q_k$, which may or may not 
be a convergent series. 
\end{notation}

\begin{theorem}[\cite{chang}, \zec\ Theorem for $\OI$] \label{thm:ZT-OI}
Given a real number $\beta\in \OI$, there is a unique $\mu \in\cFst$ such that 
$\beta=\sum_{k=1}^\infty \mu(k) \ome^k=( \mu\cdot \wh F)\ome$.
\end{theorem} 
\noindent
For the uniqueness of $\mu$ in the theorem, we require the infinite tuples such as
$ (0,1,0,1,0,\dots)$ to be not a member of $\cFst$ since
$\sum_{k=1}^\infty \ome^{2k} =\ome$, which is analogous to 
$0.0999\ldots=0.1$ in decimal expansion.

\begin{prop}[\cite{chang}]\label{prop:ZT-OI-order}
Let $\set{\mu,\tau}\subset \cFst$.
Then,
$\mu\cdot \wh F < \tau\cdot \wh F$ if and only if $\mu\vert s < \tau\vert s$ for some $s\in\nat$.
\end{prop}

Given a sequence with strong \lbd, 
we shall construct a function on $\OI$ in Definition \ref{def:f-star} below, and 
it is well-defined by Lemma \ref{lem:well-defined}.

\begin{lemma}\label{lem:well-defined}
Given a real number $\beta\in\OI$,
there is a unique $\mu\in\cFst$ such that  $\mu(1)=1$ and
 $\phi(\mu\cdot \wh F  - 1) =\beta$.
\end{lemma}
\begin{proof}
Let $\wh F^*$ be the sequence defined by $\wh F^*_n=\ome^{n}$.
Given a real number $\beta\in\OI$, we have $0<\ome +\beta\ome^2<1$.
By
Theorem \ref{thm:ZT-OI},
there are is $\mu\in\cFst$ such that 
$(\mu\cdot \wh F)\ome =\mu\cdot \wh F^*= \ome +\beta\ome^2$, which implies $\phi(\mu\cdot \wh F  - 1) =\beta$.
We claim that $\mu(1)=1$.
If $\mu(1)=0$, then by Proposition \ref{prop:ZT-OI-order},
$\ome+\beta\ome^2 = \mu \cdot \wh F^*=(0,\dots)\cdot\wh F^* < \ome=(1,0,0,\dots)\cdot \wh F^* $,
which implies a  false statement $\beta \ome^2<0$.
Thus, $\mu(1)=1$.
\end{proof}

Recall from Definition \ref{def:infinite-tuples} the definition of inequalities on tuples.
\begin{deF}\label{def:f-star}  \rm   
Let $K$ be a sequence of positive integers with strong \lbd\ under $\cF$-expansion such that
given $\mu\in\cFst$  and an integer $s\ge 2$ such that $\mu(1)=1$,
the following limit exists:
\begin{equation}
\lim_{s\to\infty}\prob{\nN : \LB_s(K_n)\le  \mu\vert s}  \label{eq:f-star}
\end{equation} 
where $\mu\vert s$ is identified in $\cF_s$.

Let $f_K^* : [0,1]\to[0,1]$ be the function given by 
$f_K^*(0)=0$, $f_K^*(1)=1$, and $
f_K^*(\phi(\mu\cdot\wh F -1))$ is equal to the value in (\ref{eq:f-star}).
 If $f_K^*$ is continuous and increasing, 
then $K$ is said to {\it have continuous \lbd\ under $\cF$-expansion}.
\end{deF}

\begin{lemma}\label{lem:finite-cf}
Let $K$ be a sequence with continuous leading block distribution under 
$\cF$-expansion, and let $f_K^*$ be the function defined
in Definition \ref{def:f-star}. 
Let $\mu\in\cFst$ 
such that there is $t\in\nat$ such that $\mu(1)=1$ and  $\mu(k)=0$ for all $k>t$. 
Then, $f_K^*(\phi( {\mu\vert t}\cdot\wh F -1))
\ \le\ \prob{\nN : \LB_t(K_n)\le \mu\vert t} 
$.
\end{lemma}

\begin{proof}
Notice that if $s>t$, then
\AAA{
&\ \set{\nN : \LB_s(K_n)\le \mu\vert s}\subset
\set{\nN : \LB_t(K_n)\le \mu\vert t}\\
\implies\  &\ 
\prob{\nN : \LB_s(K_n)\le \mu\vert s}\ \le\ 
\prob{\nN : \LB_t(K_n)\le \mu\vert t}\\
&\ \lim_{s\to\infty}
\prob{\nN : \LB_s(K_n)\le \mu\vert s}
=f_K^*(\phi( {\mu}\cdot\wh F -1))\ \le\ 
\prob{\nN : \LB_t(K_n)\le \mu\vert t}\\
\intertext{Since ${\mu\vert t}\cdot\wh F = {\mu}\cdot\wh F$,}
\implies\ &\ 
f_K^*(\phi( {\mu\vert t}\cdot\wh F -1))\ \le\ 
\prob{\nN : \LB_t(K_n)\le \mu\vert t}.
}
\end{proof}

Recall uniform continuations from Definition \ref{def:uniform-continuation}.
\begin{theorem}\label{thm:f-star}
Let $K$ be a sequence with continuous leading block distribution under 
$\cF$-expansion.
Let $f_K^*$ be the function defined in Definition 
\ref{def:f-star}.
Then, there is a uniform continuation $f$ of $F $ such that 
$f_\infty\Inv=f_K^*$ and 
$\frc{f\Inv(K_n)}$ is equidistributed.
\end{theorem} 

 \begin{proof}
Let $f : [1,\infty)\to \real $ be the function given by 
$f(x) = F_n + (F_{n+1}-F_n)(f_K^*)\Inv(p)$
where $x=n+p$ and $p=\frc{x}$.
Then, $f$ is a uniform continuation of $F_n$ since  $(f_K^*)\Inv$ is independent of $n$. Then, $f_\infty=(f_K^*)\Inv$, i.e., 
$f_\infty\Inv=f_K^*$.

Let $\beta\in(0,1)$ be a real number, and below
we show that $\prob{\nN : \frc{f\Inv(K_n)} \le \beta}$
exists, and it is equal to $\beta$.
Recall $\fF$ from Definition \ref{def:alpha}
and $\fF_n$ from Definition \ref{def:fFn}.
Let $\nN$, and let $m\in \nat$ such that $F_m \le K_n < F_{m+1}$.
Then, 
$K_n = f(m+p_n')=\fF(m+p_n)$ where $p_n,p_n'\in[0,1]$, i.e.,
$f_\infty(p_n')=\fF_m(p_n)$.
By Theorem \ref{thm:ZT-OI} and Lemma \ref{lem:well-defined}, 
there is a unique $\mu\in\cFst$ such that 
  $f_\infty(\beta) = \phi(\mu\cdot\wh F-1)$ and $\mu(1)=1$. 
  Recall $\fF_\infty$ from Definition \ref{def:fFn}.
Notice that  
\GGG{
\frc{f\Inv(K_n)}\ =\ p_n'\ \le\ \beta
\implies 
f_\infty\Inv(\fF_m(p_n)) \ \le\  \beta
\ \implies\ 
p_n\le \fF_m\Inv(f_\infty(\beta))\\
\ \implies\  \frc{\fF\Inv(K_n))}\ \le\ \fF_m\Inv(f_\infty(\beta))
\ =\ \fF_\infty\Inv(f_\infty(\beta))+o(1)
\ =\  \log_\phi(\mu\cdot \wh F) + o(1).\\
\intertext{Fix an integer $t\ge 2$.
By Proposition \ref{prop:ZT-OI-order},
we have $\mu\cdot \wh F = \mu\vert t \cdot \wh F + \gamma_t
< \wt{\mu\vert t} \cdot \wh F$ where $\gamma_t\ge 0$ and $\wt{\mu\vert t}\in \cF_t$ is as defined 
Definition \ref{def:B-tilde}.
Since $ \log_\phi( \wt{\mu\vert t}\cdot \wh F)- \log_\phi( \mu\cdot \wh F)>0$, 
there is $M_t\in\nat$ such that  for all $n\ge M_t$,}
\ \implies\ 
\frc{\fF\Inv(K_n))} \ \le\  
 \log_\phi( \mu\cdot \wh F)+ o(1)
 \ <\ 
 \log_\phi(\wt{\mu\vert t}\cdot \wh F) . \\
\intertext{By Lemma \ref{lem:LB-inequality-2}, 
this implies $\LB_t(K_n)\le  \mu\vert t$.
Recall $\Omega_n=\set{k\in\nat : k\le n}$; } 
\prob{ k\in \Omega_n : \frc{f\Inv(K_k)}  \le \beta }+o(1)
\ \le\ 
 \prob{k\in \Omega_n :\LB_t(K_k)\le \mu\vert t }+o(1)\\
 \ \implies\ 
\limsup_n \prob{ k\in \Omega_n : \frc{f\Inv(K_k)}  \le \beta } 
\ \le\ 
 \prob{\nN :\LB_t(K_n)\le \mu\vert t }. 
}

Let us work on the $\liminf$ of the probability.
Since $\beta\ne0$, there is $t_0>1$ such that 
$\mu(t_0)>0$.
Thus, if $t>t_0$ is sufficiently large, then
there are at least two entries $1$ in $\mu\vert t$, and $\mu\vert t$ has more entries after the second entry of $1$ from the left.
Recall the product $*$ from Definition \ref{def:products}.
This choice of $t$ allows us to have 
 the unique \cf s $\check \mu $ and $\wh{\mu}$   in $\cF_t$
such that $1+\check \mu * F = \wh \mu * F$
and $1+\wh \mu * F =  \mu\vert t * F$. 
Then, by Lemma \ref{lem:LB-inequality-2},
\AAA{
 &  \LB_t (K_n)\le \check \mu 
  \ \implies\  
\frc{\fF\Inv(K_n) } \ <\  \log_\phi(\wh \mu\cdot \wh F) + o(1)
\\
  \ \implies\ \ & 
p_n\ <\   \fF_m\Inv(\phi(\wh \mu\cdot \wh F-1)) +o(1) 
\\
 \ \implies\ \  & 
\fF_m(p_n)\ =\ f_\infty(p_n')\ <\   \phi(\wh \mu\cdot \wh F-1) + o(1)\\  
&
\begin{aligned}
\makebox[0ex][r]{$\ \implies\ $\ } p_n'\ =\ \frc{f\Inv(K_n)} &
   \ <\   f_\infty\Inv(\phi(\wh \mu\cdot \wh F-1)) + o(1) 
 \\ &\ <\ f_\infty\Inv(\phi( \mu\vert t \cdot \wh F-1)) 
 \quad\text{by Proposition \ref{prop:ZT-OI-order},}\\
 &  
\ \le\  f_\infty\Inv(\phi( \mu\cdot \wh F-1))  \ =\ \beta
\end{aligned}
\\
  \ \implies\ \  & 
\prob{k\in \Omega_n : \LB_t (K_k)\le \check \mu }
+o(1) \ \le\  \prob{ k\in \Omega_n : \frc{f\Inv(K_k)}  \le \beta }+o(1) \\
  \ \implies\ \  & 
\prob{\nN  : \LB_t (K_n)\le \check \mu } 
 \ \le\ \liminf_n\ \prob{ k\in \Omega_n : \frc{f\Inv(K_k)}  \le \beta }\\
 \intertext{By Lemma \ref{lem:finite-cf},}
 & f_\infty\Inv(\phi( \check \mu\cdot \wh F-1))
  \ \le\ \liminf_n\ \prob{ k\in \Omega_n : \frc{f\Inv(K_k)}  \le \beta }.
}
It is given that $\prob{\nN :\LB_t(K_n)\le \mu\vert t }\to 
f_\infty\Inv(\phi( \mu\cdot \wh F-1))$ as $t\to \infty$.
Let us calculate the other bound;
\AAA{
2+\check \mu * F &\ =\  \mu\vert t  * F
\ \implies\  \ 
2+ \sum_{k=1}^t \check \mu(k) F_{t-k+1}
= \sum_{k=1}^t  \mu (k) F_{t-k+1}\\
\ \implies\ &\ 
2+ \sum_{k=1}^t \check \mu(k)
 	\left(\al \phi^{t-k+1} + O(\phi^{-t+k-1})\right)
\ =\  \sum_{k=1}^t  \mu(k) 
	\left(\al \phi^{t-k+1} + O(\phi^{-t+k-1})\right)\\
\ \implies\ &\ 
O(1)+ \al \sum_{k=1}^t \check \mu(k)
  \phi^{t-k+1} 
\ =\ \al \sum_{k=1}^t  \mu(k)   \phi^{t-k+1}\\
\ \implies\ &\ 
O(\phi^{-t})+ \sum_{k=1}^t \check \mu(k)
 \ome^{k-1}
\ =\  \sum_{k=1}^t  \mu(k)  \ome^{k-1}\\
\ \implies\ &\ 
o(1)+ \check \mu\cdot \wh F\ =\  \mu\vert t \cdot \wh F 
\ \implies\  
\check \mu  \cdot \wh F \to \mu \cdot \wh F\\
\ \implies\ &\ 
f_\infty\Inv( \phi(\check \mu\cdot \wh F -1))
\to f_\infty\Inv( \phi(  \mu\cdot \wh F -1))\ =\ \beta.
}
\end{proof}

It is clear that if $f$ is a uniform continuation of $F$, and $K$ is 
a sequence  of positive integers approaching $\infty$ such that 
$\frc{f\Inv(K_n)}$ is equidistributed,  
then, by Lemma \ref{lem:LB-inequality-2}, $K$ has continuous \lbd\ under $\cF$-expansion.
Therefore, we have the following.
\begin{theorem}\label{thm:equivalence}
Let $K$ be a sequence of positive integers approaching $\infty$.
Then, $K$  has 
continuous \lbd\ under $\cF$-expansion
if and only if there is a uniform continuation $f$ of $F $ such that $\frc{f\Inv(K_n)}$ is equidistributed.
\end{theorem}

\section{Benford's Law under generalized \zec\ expansion }
\label{sec:general-\zec}
 
The contents in Sections \ref{sec:benford}, \ref{sec:calculations}, and 
\ref{sec:Other-continuations} are for \zec\ expansion, but
the arguments of the proofs apply to the setup for generalized \zec\ expansion
without difficulties.
In this section, we introduce definitions and results for generalized \zec\ expansion without proofs,
 but only refer to the corresponding theorems for \zec\ expansion proved in the earlier sections.

\subsection{Generalized \zec\ expansion}
 Let us review the generalized \zec\ expansion. 
 Recall $\nat_0$ from Definition \ref{def:N0}
 \begin{deF}\label{def:GZE} \rm
Given a tuple $L=(a_1,a_2,\dots,a_N)\in \nat_0^N$ where $N\ge 2$ and 
$a_1> 0$,
let $\Theta$ be the following infinite tuple in $\prod_{k=1}^\infty \nat_0$:
$$(a_1,a_2,\dots, a_{N-1}, a_N, a_1,a_2,\dots, a_{N-1}, a_N,\dots)$$
where the finite tuple $( a_1,a_2,\dots, a_{N-1}, a_N)$ repeats.
Let
$\Theta(k)$ denote the $k$th entry of $\Theta$, and let 
$\Theta\vert s = (\Theta(1),\dots,\Theta(s))$ for $s\in\nat$.

Recall $\len$ from Definition \ref{def:products}.
Let $\cH^\circ$ be the recursively-defined set of tuples $\ep$ with arbitrary  finite  length   such that 
 $\ep\in\cH^\circ$ if and only if 
 there is   smallest $s\in\nat_0$ such that $\ep\vert s = \Theta\vert s$,
 $\ep(s+1)<\Theta(s+1)$, and 
 $(\ep(s+2),\dots,\ep(n))\in \cH^\circ$ where $n=\len(\ep)$ and  $s$ is allowed to be $ \len(\ep)$.
 Let $\cH:=\set{\ep\in \cH^\circ : \ep(1)>0}$.
The set $\cH$ is called a {\it periodic \zec\ collection of \cf s for positive integers},
and $L$ is called {\it a principal maximal block of the periodic \zec\ collection $\cH$}.
\end{deF} 
\noindent
Notice that if $L=(1,0,1,0)$ is a principal maximal block of the periodic \zec\ collection $\cH$,
then  $L'=(1,0)$ is a principal maximal block of $\cH$ as well.
For this reason, the indefinite article was used in the statement of the definition of principal maximal blocks.
\begin{example}\rm
Let $\cH$ be the (periodic) \zec\ collection determined by the principal maximal block
  $L=(3,2,1)$.  Then, $\Theta=(3,2,1,3,2,1,\dots)$, and  $(0)$ and $(3,2,1)$ are members of $\cH^\circ$.
  For $(0)\in \cH^\circ$, we set $s=0$ in Definition \ref{def:GZE},
  and for $(3,2,1)\in\cH^\circ$, we set $s=3$.
  
Let $\ep=(3,2,0)$ and $\mu=(3,1,3,2,0)$.
For $\ep$, if $s=2$, by the definition, we have $\ep\in\cH$.
For $\mu$, if $s=1$, then $\mu\vert 1 = \Theta\vert 1$,
 $\mu(2)<\Theta(2)$, and 
 $(\mu(3),\dots,\mu(5))=\ep\in \cH^\circ$.
 Listed below are more examples of members of $\cH$:
 $$
(3,2,1,3,2,1),\ (3,0,0,3),\ (1,2,3,1,0,3),\ (1,2,3,1,1,0). $$
\end{example}

Recall the product notation from Definition \ref{def:products}
\begin{deF}
\rm \label{def:FS}
Let $\cH$ be a set of \cf s, and let $H$ be an increasing sequence of positive integers.
If given $n\in \nat$, there is a unique $\ep\in\cH$ such that $\ep * H=n$,
then $H$ is called a {\it \funds} of $\cH$, and the expression $\ep * H$ is called an {\it $\cH$-expansion}.
\end{deF}
If $\cH$ is a periodic \zec\ collection for positive integers,
then, by Theorem \ref{thm:funds} below,
   there is a unique \funds\ of $\cH$.
\begin{theorem}[\cite{chang,mw}] \label{thm:funds}
Let $\cH$ be a periodic \zec\ collection, and let $L=(a_1,\dots,a_N)$ be its principal maximal block.
Then, there is a unique \funds\ $H$ of $\cH$, and it is given by the following recursion:
\begin{gather}
H_{n+N}\ =\ a_1 H_{n+N-1} +\cdots 
+ a_{N-1} H_{n+1}+ (1+a_N) H_n\ \text{ for all $\nN$, and }\label{eq:H-recursion} \\
 H_n\ =\ 1+\sum_{k=1}^{n-1} a_k H_{n-k}\ \text{ for all $1\le n\le N+1$.}\notag
\end{gather} 
\end{theorem}
\noindent
If $L=(1,0)$, then its periodic \zec\ collection is $\cF$ defined in Definition \ref{def:ome}, and
its \funds\ is the \fib\ sequence.
If $L=(9,9)$, then the \funds\ $H$ is given by $H_n=10^{n-1}$,  and $\ep* H$ for $\ep\in\cH$ are base-$10$ expansions. 

\begin{deF}
\rm \label{def:dominant-zero}
Let $L=(a_1,\dots,a_N)$ be the list defined in Definition \ref{def:GZE}. 
Let $\psi=\psi_{\cH}=\psi_L$ be the dominant real zero 
of the polynomial $g=g_{\cH}=g_L(x):=x^N - \sum_{k=1}^{N-1} a_k x^{N-k}- (1+a_N)$, and 
$\theta:=\psi^{-1}$.
Let $\wh H$ be the sequence given by $\wh H_n=\theta^{n-1} $.
\end{deF}
\noindent
By (\ref{eq:H-recursion}),
 the sequence $\wh H$ in Definition \ref{def:dominant-zero} satisfies 
\begin{equation}\label{eq:theta}
 \wh H_n\ =\ a_1 \wh H_{n+1} +\cdots + a_{N-1} \wh H_{n+N-1}+(1+ a_N) \wh H_{n+N}\quad
 \text{for all $\nN$}.
\end{equation} 

The following proposition is proved in \cite[Lemma 43]{chang} and \cite[Lemma 2.1]{martinez}.
\begin{prop}
\label{prop:dominant-zero}
Let $L=(a_1,\dots,a_N)$ be the list defined in Definition \ref{def:GZE}, and let
$g=x^N - \sum_{k=1}^{N-1} a_k x^{N-k}- (1+a_N)$ be the polynomial.
Then, $g$ has one and only one positive real zero $\psi$,   it is a simple zero, and 
there are no other complex zeros $z$ such that $\abs{z}\ge \psi$.
\end{prop}

\begin{theorem}
\label{thm:binet}
Let $\cH$ be a periodic \zec\ collection with a principal maximal block $L=(a_1,\dots,a_N)$, and 
let $H$ be the \funds\ of $\cH$.  Then $H_n = \delta \psi^n + O(\psi^{ rn})$ for $n\in \nat$
where $\delta$ and $r$  are  positive (real) constants, $r<1$, and   $\psi$ is the dominant zero defined in Definition \ref{def:dominant-zero}.
\end{theorem}

\begin{proof}
Let $g$ be the characteristic polynomial of degree $N$ defined in Definition \ref{def:dominant-zero}, and 
let $\set{\lambda_1,\dots, \lambda_m}$ be the set of $m$ distinct (complex) zeros of $g$ where $m\le N$ and $\lambda_1=\psi$.
Then, by Proposition \ref{prop:dominant-zero}, 
 we have $\abs{\lambda_k}< \psi$ for $2\le k\le m$.
Since $\psi$ is a simple zero, by the generalized Binet's formula \cite{levesque}, there are polynomials $h_k$ for $2\le k\le m$ and a constant $\delta$ such that 
$H_n = \delta \psi^n + \sum_{k=2}^m h_k(n) \lambda_k^n$ for $\nN$.
Thus, there is a positive real number $r<1$ such that $H_n = \delta \psi^n + O(\psi^{rn})$ for $\nN$.

Notice that $\lim_{n\to\infty}  H_n/\psi^n= \delta$, and
let us show that $\delta$ is a positive real number, and in particular, it is non-zero.
By \cite[Theorem 5.1]{chang-2023}, 
\begin{align}
 \delta\ =\ \lim_{n\to\infty} \frac{H_n}{\psi^{n }} &\ =\ 
 \frac{1}{\psi g'(\psi)} \sum_{k=1}^N \frac{ H_k}{ (k-1)!}
 \left[ \frac{ d^{k-1}}{dx^{k-1}} \frac{g(x)}{x-\psi} \right]_{x=0}. \label{eq:alpha}\\ 
 \intertext{By the product rule, we have}
 \left[ \frac{ d^{k-1}}{dx^{k-1}} \frac{g(x)}{x-\psi} \right]_{x=0}
 &\ =\  \left[ \sum_{j=0}^{k-1} \binom{k-1}j g^{(j)}(x)\, (x-\psi)^{-1-j} \prod_{t=1}^{j}(-t)\right]_{x=0}.\notag
\end{align} 
Notice that  if $1\le j\le N-1$, then $g^{(j)}(0)= -a_{N-j} j!\le 0$,
and if $g(0)=-(1+a_N)<0$. 
The inequality $(-\psi)^{-1-j} \prod_{t=1}^{j}(-t)<0 $ for all $0\le j\le k-1$
follows immediately from considering the cases of $j$ being even or odd.
Thus, the summands in (\ref{eq:alpha}) are non-negative, and some   are positive.
This concludes the proof of $\delta$ being a positive real number.
\end{proof}

\noindent
For the remainder of the paper, 
let  $\cH$, $H$, and $\psi$ be as defined in Definition \ref{def:GZE}.

\subsection{Strong \benlaw}

Let us begin with definitions related to leading blocks under $\cH$-expansion.

 \begin{deF}\label{def:GZ-LB}
\rm
Let $n=\ep * H$ for $\nN$ and $\ep\in\cH$.
If $s\le \len(\ep)$, then  
 $ (\ep(1),\dots,\ep(s))\in\cH$ is called
 {\it the leading  block of $n$ with length $s$ under  $\cH$-expansion}.
    Recall that $N=\len(L)$.
If $N\le s \le \len(\ep)$, let 
  $\LB_s^{\cH}(n)$, or simply  $\LB_s(n)$ if the context is clear, denote the leading block of length $s$,
  and if $s\le\len(\ep)$ and $s<N$,
  then let $\LB_s^{\cH}(n)$  or simply  $\LB_s(n)$ denote
  $(\ep(1),\dots,\ep(s),0,\dots,0)\in \nat_0^N$.
 If $s>\len(\ep)$, $\LB_s(n)$ is declared to be undefined. 
  
Recall the product $*$ from Definition \ref{def:products}.
Given an integer $s\ge N$, let
 $\cH_s :=\set{\vecb_1,\vecb_2,\dots, \vecb_\ell}$ be the finite set of 
 the leading  blocks of length $s$ occurring in the $\cH$-expansions of $\nat$ such that
  $1+\vecb_k*H=\vecb_{k+1}*H$ for all $k\le \ell-1$. 
  Recall the truncation notation from Definition \ref{def:vert}.
  If $1\le s<N$, then  
let
 $\cH_s :=\set{\vecb_1,\vecb_2,\dots, \vecb_\ell}$ be the finite set of 
 the leading  blocks of length $N$ occurring in the $\cH$-expansions of $\nat$ such that
 $\vecb_k(j)=0$ for all $1\le k\le \ell$ and $j>s$ and 
  $1+\vecb_k|s*H=\vecb_{k+1}|s*H$ for all $k\le \ell-1$. The leading block $\vecb_\ell$ is called 
  {\it the largest leading block in $\cH_s$}. 
     
 The exclusive block $\vecb_{\ell+1}$ is a \cf\ of length $s$ defined as follows.
 If $s\ge N$, $s\equiv p \moD N$, and $0\le p<N$, then
 \GGG{ 
\vecb_{\ell+1}: =(a_1,\dots,a_{N-1},  a_N,\dots, a_1,\dots,a_{N-1},1+ a_N, c_1,\dots,c_p)
 }  
 where   $c_k=0$ for all $k$.
 If $1\le s<N$, then $
\vecb_{\ell+1} :=( a_1,\dots,a_{N-1},1+ a_N )$.
 If $\vecb$ is a leading block $\vecb_k\in \cH_s$, then we denote $\vecb_{k+1}$ by $\wt \vecb$.
\end{deF}

 If $s<N$, then the leading blocks $\vecb$  in $\cH_s$ has lengths $N$ with $N-s$ last entries of $0$, and 
 this case is treated as above in order to make $\vecb$ and $\wt \vecb$ in the statement and proof of Lemma \ref{lem:LB-inequality-2}
 fit into the case of periodic \zec\ collections; see Lemma \ref{lem:LB-inequality-3}.

By \cite[Definition 2 \&\ Lemma 3]{chang} and Theorem \ref{thm:funds},
the subscript numbering of $\vecb_k\in\cH_s $ for $1\le k\le \ell$ coincides with 
the \lex\ order on the \cf s.
If
 $\vecb$ is the largest leading block in $\cH_s$ where $s\ge N$,
 then
 \AAA{ 
\vecb &=(\dots, a_1,\dots,  a_N,a_1,\dots,a_p)
 \text{  
 if $s\equiv p \moD N$ and $0\le p<N$,}
 }
and $1+ \vecb  * H = \wt \vecb *H=(\dots, a_1,\dots,1+  a_N,0,\dots,0)*H=H_{s+1}$
where the last  $p$ entries of $\wt \vecb$ are zeros.
If $s\equiv 0\moD N$ and $\vecb$ is the largest leading block in $\cH_s$,
then 
$$\wt \vecb= (a_1,\dots,a_{N-1},  a_N,\dots, a_1,\dots,a_{N-1},1+ a_N).$$
 If $s<N$ and $\vecb$ is the largest leading block in $\cH_s$, then $\wt \vecb = ( a_1,\dots,a_{N-1},1+ a_N )$. 
 Recall $\wh H$ from Definition \ref{def:dominant-zero}.  For all cases, if $\vecb$ is the largest leading block in $\cF_s$, then
$ \wt \vecb \cdot \wh H =\psi$.

The proof of Theorem \ref{thm:equidistribution-examples-2} below follows immediately from
Lemma \ref{lem:fH-inv} and  Theorem \ref{thm:equidistr-BL-H}.

\begin{theorem}\label{thm:equidistribution-examples-2}
Let $K$ be a sequence of positive integers such that 
 $K_n= a b^n(1+o(1))$ where $a$ and $ b$ are positive real numbers such that $\log_\psi b$ is irrational.
Then, given  $\vecb\in \cH_s $,
$$\prob{\nN : \LB_{s}(K_n) = \vecb } 
\ =\  \log_\psi\frac{\wt \vecb \cdot \wh H}{  \vecb \cdot \wh H}.$$ 
\end{theorem} 

Motivated from the leading block distributions of the exponential sequences considered in 
Theorem \ref{thm:equidistribution-examples-2},
we declare strong \benlaw\ under $\cH$-expansion as follows.
\begin{deF}\label{def:BL-GZ} \rm 
A  sequence $K$   of positive integers is said to {\it satisfy strong Benford's Law  under  $\cH$-expansion}
if given   $\vecb\in \cH_s $, 
$$\prob{\nN : \LB_s(K_n)=\vecb  } \ =\  \log_{\psi}\frac{\wt \vecb\cdot \wh H}
{\vecb \cdot \wh H}. $$   
\end{deF}

\subsection{Benford  continuation of $H$}
We used a real analytic continuation of the \fib\ sequence for \zec\ expansion,
but as demonstrated in the earlier sections,  the leading block distributions are determined by
  its limit $\fF_\infty$.
Thus, rather than using a real analytic continuation of $H$, 
we may use the limit version directly, which is far more convenient.
By Theorem \ref{thm:binet}, $H_n =\delta \psi^n + O(\psi^{rn})=\delta \psi^n(1+o(1))$ where $\delta$ and $r<1$ are positive real constants, and 
we define the following:

\begin{deF}\rm \label{def:B-continuation}
Let $\fH : [1,\infty) \to \real$ be the function given by 
$$\fH(x)=H_n + (H_{n+1}-H_n)  \frac{\psi^p-1}{\psi-1}$$
where $x=n+p$ and $p=\frc{x}$, and 
it is called {\it a Benford continuation of $H$}.
\end{deF}
\noindent
Recall Definition \ref{def:uniform-continuation}.
Then, $\fH$ is a uniform continuation of $H$, and $\fH_\infty(p) =\frac{\psi^p-1}{\psi-1}$ for all $p\in[0,1]$.
We leave the proof of the following to the reader.
\begin{lemma} \label{lem:fH-inv}
For real numbers $x\in[1,\infty)$, we have
$\fH(x) = \delta \psi^x(1+o(1))$, and $\fH\Inv(x) =\log_\psi(x) - \log_\psi \delta + o(1)$.
\end{lemma}

Recall $\cH_s$ from Definition \ref{def:GZ-LB} and $\wh H$ from Definition \ref{def:dominant-zero}.
\begin{lemma}\label{lem:LB-inequality-3}
Let $K$ be a sequence of positive real numbers approaching $\infty$.
Let $\vecb\in\cH_s$, and
  let  $A_\vecb:=\set{\nN : \LB_{s}(K_n)=\vecb }$.
Then, there are real numbers $\gamma_n=o(1)$ and $\wt \gamma_n=o(1)$
such that 
  $n\in A_\vecb$ if and only if
\begin{equation}
\log_\psi
  	  \vecb \cdot \wh H +\gamma_n \ \le\ \frc{\fH\Inv(K_n)}\ <\  \log_\psi
  	  \wt \vecb\cdot \wh H + \wt\gamma_n, \label{eq:BL-inequality-H}
\end{equation}  
  	  where $ \wt\gamma_n=0$ when $\vecb$ is the largest leading block of length $s$.
\end{lemma}
\noindent
There is no difficulty in applying the arguments of the proof of Lemma \ref{lem:LB-inequality-2} 
to Lemma \ref{lem:LB-inequality-3}, and we leave the proof to the reader.

 Recall Definition \ref{def:BL-GZ}.
\begin{theorem}\label{thm:equidistr-BL-H}
Let $K $ be an increasing  sequence of positive integers such that 
 $\frc{\fH\Inv(K_n)}$  is equidistributed.
Then, $K$ satisfies strong Benford's Law under the $\cH$-expansion.  
\end{theorem} 
 \noindent
There is no difficulty in applying the arguments of the proof of Theorem \ref{thm:equidistr-BL}
to Theorem \ref{thm:equidistr-BL-H}, and we leave the proof to the reader.

\subsection{Absolute \benlaw}\label{sec:absolute}

Introduced in \cite{chang} is  a full generalization of \zec\ expressions, which is based on the very principle of how \zec\ expressions are constructed in terms of \lex\ order.  In this most general sense, the collection $\cH$ in Definition \ref{def:GZE} is called
  a periodic \zec\ collection of \cf s.
We believe that a property concerning all periodic \zec\ collections may be noteworthy,
and as in  the notion of normal numbers, 
  we introduce the following definition. 
\begin{deF}\label{def:BL-absolute}
\rm
A sequence $K$ of positive integers is said to 
{\it satisfy absolute \benlaw}
if  $K$ satisfies strong $\cH$-\benlaw\ for each periodic \zec\ collection  $\cH$.
\end{deF}
Recall the Lucas sequence $K=(2,1,3,4,\dots)$ from Example 
\ref{exm:lucas}.
It satisfies strong \benlaw\ under all base-$b$ expansions, 
but it does not satisfy strong \benlaw\ under \zec\ expansion.
Thus, the Lucas sequence does not satisfy absolute \benlaw.

\begin{theorem}\label{thm:gamma-n}
Let $\gamma$ be a positive real number such that 
 $\gamma$ is not equal to $\psi^r$ for any $r\in \ratn$ and 
   any dominant real zero $\psi$ of $g_\cH$ where $\cH$ is as defined in Definition \ref{def:dominant-zero}.
   Let $K$ be the sequence given by $K_n= \flr{\gamma^{n } } $.
Then, $K$ satisfies absolute \benlaw.
\end{theorem} 

\begin{proof}
Let $H$ and $\psi$ be as defined in Definitions \ref{def:FS} and \ref{def:dominant-zero}, and let $\fH$ be the Benford continuation defined in Definition \ref{def:B-continuation}.  Note that $\psi$ is algebraic.
Notice that  $\flr{\gamma^{n } }=\gamma^{n+o(1)} $, and  $\log_\psi(\gamma)$ is irrational.
Thus, by Lemma \ref{lem:fH-inv},
\GGG{
\fH\Inv(K_n)\ =\ (n + o(1)) \log_\psi(\gamma) -\log_\psi(\delta) + o(1)
\ =\ n \log_\psi(\gamma) -\log_\psi(\delta) + o(1).\\
\intertext{By Weyl's Equidistribution Theorem,}
\ \implies\  
\prob{\nN : \frc{\fH\Inv(K_n)}\le \beta} 
\ =\ \prob{\nN : \frc{n \log_\psi(\gamma) }\le \beta}\ =\ \beta.
}
By Theorem \ref{thm:equidistr-BL-H}, $K$ satisfies \benlaw\ under $\cH$-expansion.
\end{proof}

\begin{cor}\label{cor:gamma-n}
Let $\gamma>1$ be a real number that is not an algebraic integer.
Then, the sequence $K$ given by $K_n= \flr{\gamma^{n } } $ satisfies
absolute \benlaw.
\end{cor}

\begin{proof}
The dominant real zero $\psi$ defined in Definition \ref{def:dominant-zero} is 
an algebraic integer, and so is $\psi^r$ for all $r\in\ratn$.
Thus, if $\gamma\in \real$ is not an algebraic integer, then by Theorem \ref{thm:gamma-n}, 
$K$ satisfies absolute \benlaw.
\end{proof}

\begin{example} \label{exm:ABL}
\rm
Let
$K$ be the sequence given by $K_n=  \flr{ \frac{\phi}{\sqrt5}(\tfrac{89}{55})^{n } } $, which is considered in the introduction.
Since $\tfrac{89}{55}$ is not an algebraic integer, by Corollary \ref{cor:gamma-n}, 
  the sequence $K$ satisfies absolute \benlaw. 

\end{example}
 
 \subsection{Other Continuations}
Recall Definition \ref{def:uniform-continuation}, and that $H$ is the fundamental sequence of $\cH$ defined in
Definition \ref{def:FS}.
 As in Section \ref{sec:Other-continuations}, 
 we relate other continuations of $H$ to the distributions of leading blocks under $\cH$-expansion.

Recall the Benford continuation $\fH$ from Definition \ref{def:B-continuation},
uniform continuations $h$ and $h_\infty$ from Definition \ref{def:uniform-continuation}, and
the definition of $\wt\vecb$ from Definition \ref{def:GZ-LB}. 
\begin{theorem}\label{thm:continuation-H}
Let $h : [1,\infty) \to \real$ be  a  uniform continuation of $H$. 
    Then,  there is a  sequence
    $K$ of positive integers approaching $\infty$, e.g.,
 $K_n=  \flr{\fH(n + \fH_n\Inv\circ h_n \big( \frc{n\pi} \big)}  $, such that  
    $\frc{ h\Inv(K_n) }$ is equidistributed. 

Let $K$ be a sequence of of positive integers approaching $\infty$ such that  
    $\frc{h\Inv(K_n) }$ is equidistributed.
Let $\vecb\in \cH_s$.
Then, 
\begin{align}
\prob{\nN : \LB_{s}( {K_n}) =\vecb }  
 &\ =\  
h_\infty\Inv
\circ \fH_\infty(\log_\psi\wt  \vecb\cdot \wh H )  - 
h_\infty\Inv\circ   \fH_\infty(\log_\psi \vecb\cdot \wh H ) \notag
 \\
 &\ =\ 
h_\infty\Inv
\left(\frac{\wt  \vecb\cdot \wh H - 1}{\psi - 1} \right) - 
h_\infty\Inv\left( \frac{ \vecb\cdot \wh H - 1}{\psi - 1}\right).
\notag 
\end{align}

\end{theorem}
\noindent
There is no difficulty in applying the arguments of the proof of Theorem \ref{thm:continuation}
to Theorem \ref{thm:continuation-H}, and we leave the proof to the reader.

Recall that $\OI=(0,1)$.
As in Definition \ref{def:infinite-tuples}, 
we introduce expressions for $\OI$ that are associated with $\cH$.
Recall also  the infinite tuple $\Theta$, $\theta$, and  $\wh H$,  from Definitions \ref{def:GZE}  and \ref{def:dominant-zero}.
\begin{deF}\label{def:infinite-tuples-H}\rm

An infinite tuple $ \mu\in\prod_{k=1}^\infty \nat_0$ is 
called an {\it $\cH$-expression for $\OI $} if
there is a smallest $i\in\nat$ such that $\mu(i)>0$,
$(\mu(i),\dots,\mu(k))\in \cH$ for all $k\ge i$, and for all $j\in\nat_0$,
the sequence $\seQ{\mu(j+n)}$ is not equal to the sequence  $ \seQ{\Theta(n)}$.
Let $\cHst$ be the set of $\cH$-expressions for $\OI$.

Given $s\in\nat$ and  $\set{\mu,\tau}\subset\cHst$, we declare $\mu\vert s<\tau\vert s$ if $\mu\vert s \cdot \wh H
< \tau\vert s\cdot\wh H$, which coincides with the \lex\ order on $\nat_0^s$.
We define $\mu\cdot \wh H:=\sum_{k=1}^\infty \mu(k)\theta^{k-1}$, which is a convergent series. 
 
\end{deF}

Theorem \ref{thm:ZT-OI-H} and Proposition \ref{prop:ZT-OI-order-H} below are proved in \cite{chang}.
\begin{theorem}[\zec\ Theorem for $\OI$] \label{thm:ZT-OI-H}
Given a real number $\beta\in \OI$, there is a unique $\mu \in\cHst$ such that 
$\beta=\sum_{k=1}^\infty \mu(k) \theta^k=( \mu\cdot \wh H)\theta$.
\end{theorem}

\begin{prop}\label{prop:ZT-OI-order-H}
Let $\set{\mu,\tau}\subset \cHst$.
Then,
$\mu\cdot \wh H < \tau\cdot \wh H$ if and only if $\mu\vert s < \tau\vert s$ for some $s\in\nat$.
\end{prop}

 By Theorem \ref{thm:ZT-OI-H}, Proposition \ref{prop:ZT-OI-order-H} and  (\ref{eq:theta}), 
 the  function from $\set{\mu\in\cF^* : \mu(1)=1}$ to $ [0,1)$ given by the following is bijective:
 $$ \mu \mapsto \frac{ \mu\cdot \wh H - 1}{\psi - 1},$$
 and hence, 
  $h_K^*$ defined in Definition \ref{def:h-star} is well-defined.

\begin{deF}\label{def:h-star}  \rm  
Let $K$ be a sequence of positive integers approaching $\infty$  such that
given $\mu\in\cHst$   such that $\mu(1)=1$,
the following limit exists:
\begin{equation}
\lim_{s\to\infty}\prob{\nN : \LB_s(K_n)\le  \mu\vert s}. \label{eq:h-star}
\end{equation} 

Let $h_K^* : [0,1]\to[0,1]$ be the function given by 
$h_K^*(0)=0$, $h_K^*(1)=1$, and $
h_K^*\left( \frac{ \mu\cdot \wh H - 1}{\psi - 1}\right)$ is equal to the value in (\ref{eq:h-star}).
 If $h_K^*$ is continuous and increasing, 
then $K$ is said to {\it have continuous \lbd\ under $\cH$-expansion}.
\end{deF}

\begin{theorem}\label{thm:h-star}
Let $K$ be a sequence with continuous leading block distribution under 
$\cH$-expansion.
Let $h_K^*$ be the function defined in Definition 
\ref{def:h-star}.
Then, there is a uniform continuation $h$ of $H_n$ such that 
$h_\infty\Inv=h_K^*$ and 
$\frc{h\Inv(K_n)}$ is equidistributed.
\end{theorem} 
 \noindent
There is no difficulty in applying the arguments of the proof of Theorem \ref{thm:f-star}
to Theorem \ref{thm:h-star}, and we leave the proof to the reader.

\section{Benford behavior within expansions}\label{sec:within-expansion}

As mentioned in the introduction, \benlaw\ under base-$b$ expansion arises with \zec\ expansion, and 
let us review this result, which is available in \cite{miller-behavior-2}.

Let $\cK$ be a periodic \zec\ collection defined in Definition \ref{def:GZE}, and 
let $K$ be the  fundamental sequence of $\cK$, defined in Definition \ref{def:FS}.
Let $S$ be an infinite subset of $\set{K_n : \nN}$ 
such that $q(S):=\prob{\nN : K_n\in S}$ exists. 
Recall the product $*$ from Definition \ref{def:products}.
For a randomly selected integer $n\in [1,K_{t+1})$,  let 
$\mu * K$ be the $\cK$-expansion of $n$, let $M=\len(\mu)$, and define
\begin{equation}
P_t (n):=\frac{\sum_{k=1}^M \mu(k)\chi_S(K_k)}
 {\sum_{k=1}^M \mu(k)} \label{eq:Psmu}
\end{equation}  
where $\chi_S$ is the characteristic function 
on $\set{ K_k : k\in \nat}$, i.e., $\chi_S(K_k)=1$ if $K_k\in S$ and $\chi_S(K_k)=0$, otherwise.
Proved in \cite{miller-behavior} is that given a  real number 
$\ep>0$, 
the probability of $n\in [1,K_{t+1})$
such that $\abs{P_t (n)-q(S)}< \ep  $ is equal to $1+o(1)$ 
as a function of $t$.
For Benford behavior, we 
let $S$  be the set of  $K_n$ that have leading (fixed)
leading decimal digit $d$.
Then, $q(S)=\log_{10}(1+\frac1d)$, and 
the probability of having a summand $K_n$ with leading digit $d$ 
within the $\cK$-expansion is nearly $q(S)$ most of the times. 

This result immediately applies to our setup.
Let $\cH$ and $H$ be as defined in Definition \ref{def:GZE} different from $\cK$ and $K$.
For example, let $\cH$ be the base-$b$ expressions, and 
let $\cK$ be the \zec\ expressions.
Then,
$H$ is the sequence given by $H_n=b^{n-1} $ and $K=F$ is the \fib\ sequence.
Recall from Definition \ref{def:GZ-LB} that $\cH_s $ is a set of leading blocks under $\cH$-expansion, and  
 that $\LB_s^\cH(n) $ denotes the leading block of $n$ in $\cH_s$ under $\cH$-expansion.
By Corollary \ref{cor:equidistribution-examples}, the sequence $K$ satisfies
(strong) Benford's Law under $\cH$-expansion, i.e., 
$$\prob{\nN : \LB_s^\cH(K_n)=\vecb }\ =\ 
\log_\psi \frac{ \wt \vecb\cdot \wh H}{ \vecb\cdot \wh H}
$$
where 
$\vecb\in\cH_s$ and $\psi=b$, and 
this is  \benlaw\ under base-$b$ expansion.
The case considered in the introduction is that
$\cH$ is the \zec\ expansion and $\cK$ is the binary expansion.
The following is a corollary of \cite[Theorem 1.1]{miller-behavior-2}.
Recall Definition \ref{def:dominant-zero}.
\begin{theorem}\label{thm:Psmu}
Let $\cH$ and $H$ be as defined in Definition \ref{def:GZE}, and 
Let $K$ be the fundamental sequence of a   periodic \zec\ collection $\mathcal K$ such that 
$\psi_\cH^r \ne \psi_\cK $ for all $r\in\ratn$ where 
$\psi_\cH$ and $\psi_\cK$ are the dominant real zeros of $g_\cH$ and $g_\cK$, respectively.
Given $\vecb\in \cH_s$, 
let $S_\vecb:=\set{K_n : \LB_s^\cH(K_n) = \vecb,\ \nN }$.
For a randomly selected integer $n\in [1,K_{t+1})$,  let  
  $P_t(n)$ be the proportion defined in (\ref{eq:Psmu}) with respect to $S=S_\vecb$.
Then, given a   real number 
$\ep>0$, 
the probability of $n\in [1,K_{t+1})$
such that 
$$\abs{P_t(n)-\log_{\psi_\cH}\frac{ \wt \vecb\cdot \wh H}{ \vecb\cdot \wh H}}\ <\  \ep $$ is equal to $1+o(1)$
as a function of $t$.
\end{theorem} 

\section{Future work}

Instead of the leading digit, one can look at the distribution of the
digit in the second, third, or generally any location. For a sequence that
is strong Benford, the further to the right we move in location, the more
uniform is the distribution of digits. A natural question is to ask
whether or not a similar phenomenon happens with Zeckendorf
decompositions, especially as there is a natural furthest to the right one
can move.

We can also look at signed Zeckendorf decompositions. Alpert \cite{A}
proved that every integer can be written uniquely as a sum of Fibonacci
numbers and their additive inverses where two if two consecutive summands
have the same sign then their indices differ by at least 4 and if they are
of opposite sign then their indices differ by at least 3. We now have more
possibilities for the leading block, and one can ask about the various
probabilities. More generally, one can consider the $f$-decompositions
introduced in \cite{DDKMMV}, or the non-periodic \zec\ collections introduced in \cite{chang}.

Additionally, one can explore sequences where there is no longer a unique
decomposition, see for example \cite{ BHLLMT1, BHLLMT2,CFHMN1, CFHMN2, chang-2018},
and ask what is the distribution of possible leading blocks. There are
many ways we can formulate this question. We could look at all legal
decompositions, we could look at what happens for specific numbers, we
could look at what
happens for specific types of decompositions, such as those arising from
the greedy algorithm or those that use the
fewest or most summands.

\end{document}